\documentclass[12pt]{amsart}
\usepackage[pagewise]{lineno}
\usepackage{amssymb,amsthm,amsmath}
\usepackage{mathtools}
\usepackage{amsfonts}
\usepackage{amscd}
\usepackage{amssymb}

\usepackage{enumerate}
\usepackage{bbm}
\usepackage{graphicx}
\allowdisplaybreaks
\usepackage{mathabx}
\usepackage{color}
\usepackage{amsbsy}
\usepackage{graphicx}
\usepackage{amsthm}
\usepackage{amsmath}
\usepackage{amsxtra}
\usepackage{mathrsfs}
\usepackage{bbm}
\usepackage{dsfont}
\usepackage{enumitem}
\usepackage{comment}
\usepackage{relsize}
\usepackage{enumerate}
\usepackage{xcolor}
\usepackage{float} 
\usepackage{url}
\usepackage{soul}
\usepackage[hidelinks]{hyperref} 
\usepackage{ifthen}

\newtheorem{theorem}{Theorem}[section]
\newtheorem{lemma}[theorem]{Lemma}
\newtheorem{corollary}[theorem]{Corollary}
\newtheorem{proposition}[theorem]{Proposition}

\theoremstyle{definition}

\newtheorem{remark} [theorem]{Remark}

\theoremstyle{remark}

\numberwithin{equation}{section}

\newcommand{\C}{\mathbb{C}}

\newcommand{\R}{\mathbb{R}}

\newcommand{\Z}{\mathbb{Z}}

\newcommand{\fg}{\mathfrak{g}}

\newcommand{\fl}{\mathfrak{l}}

\newcommand{\cL}{\mathcal{L}}
\newcommand{\cK}{\mathcal{K}}

 \newcommand{\what}{\widehat}

\newcommand{\beas}{\begin{eqnarray*}}
\newcommand{\eeas}{\end{eqnarray*}}
\newcommand{\bes} {\begin{equation*}}
\newcommand{\ees} {\end{equation*}}
\newcommand{\be} {\begin{equation}}
\newcommand{\ee} {\end{equation}}
\newcommand{\bea} {\begin{eqnarray}}
\newcommand{\eea} {\end{eqnarray}}

\let\oldproofname=\proofname
\renewcommand{\proofname}{\rm\bf{\oldproofname}}

\newcommand{\ol}{\overline}
\newcommand{\mf}{\mathfrak}

\title[Spherical and lacunary maximal operators]{Spherical and lacunary maximal operators on rank one Riemannian symmetric spaces of noncompact type}
\author{Subir Dakshi and Sanjoy Pusti}

\address{Subir Dakshi \endgraf Department of Mathematics, \endgraf INDIAN INSTITUTE OF TECHNOLOGY BOMBAY, \endgraf Powai, Mumbai-400076, India.}
\email{subirdakshi@iitb.ac.in }
\address{Sanjoy Pusti \endgraf Department of Mathematics, \endgraf INDIAN INSTITUTE OF TECHNOLOGY BOMBAY, \endgraf Powai, Mumbai-400076, India.}
\email{sanjoy@math.iitb.ac.in}


\begin{document}

\begin{abstract}
  The spherical maximal operator of order $\mu$ was introduced by El Kohen in the setting of real hyperbolic spaces, where its $L^p$-boundedness properties were investigated. In this paper, we extend this operator to rank-one Riemannian symmetric spaces of noncompact type and establish sufficient conditions on $\mu$ for its boundedness on $L^p$, $1<p\leq\infty$. We also obtain a necessary condition for $L^p$-boundedness when $\mu$ is real and $1<p<\infty$. For $1<p\leq 2$, the necessary and sufficient conditions coincide, yielding the sharp range of real $\mu$ in this regime. Finally, we establish sufficient conditions for the $L^p$-boundedness of the corresponding lacunary spherical maximal operator for $1<p\leq\infty$.
\end{abstract}

\maketitle
\begin{center}
    {\em The paper is dedicated to the loving memory of Dr. Jayanta Sarkar.}
\end{center}
\section{Introduction}
Spherical maximal functions constitute one of the fundamental examples of maximal operators associated with curved hypersurfaces. 
For $\operatorname{Re}\alpha>0$, Stein (\cite{Stein1976}) defined the spherical maximal function of order $\alpha$ as, \bes M^\alpha f(x)=\sup_{t>0}| M_{t}^\alpha f(x)|, \ees
where
\bes  M_{t}^\alpha f(x)=\frac{1}{\Gamma(\alpha)}
t^{-n}
\int_{|y|<t}
\left(1-\frac{|y|^2}{t^2}\right)^{\alpha-1}
f(x-y)\,dy.
\ees
The operators $M_{t}^\alpha$ admit an analytic continuation in the parameter $\alpha$. At $\alpha=0$, the analytically continued family recovers, up to normalization, the ordinary spherical mean operator \bes M^0f(x)=\sup_{t>0}|f*\sigma_t(x)|=\sup_{t>0}\left|\int_{S^{n-1}}f(x-ty)\,d\sigma(y)\right|,\ees where $d\sigma$ denotes the normalized surface measure on the unit sphere \(S^{n-1}\subset\mathbb R^n\). Also, at $\alpha=1$, this family reduces to the Hardy-Littlewood maximal function. 
Stein proved 
\begin{theorem}[Stein, \cite{Stein1976}]
    For $n\geq 2$, \bes 
    \|M^\alpha f\|_{L^p(\mathbb R^n)} \leq C \|f\|_{L^p(\mathbb R^n)},
    \ees
    for $\alpha>1-n+\frac{n}{p}$ when $1<p\leq 2$ or for $\alpha>\frac{2-n}{p}$ when $2\leq p \leq \infty$.
\end{theorem}
Therefore, for $\alpha=0$ and for $n\geq 3$, the spherical maximal function $M^0f$ is bounded for $p>\frac{n}{n-1}$. Stein proved that this range is sharp. Bourgain (\cite{Bourgain1986}) proved that this maximal function $M^0f$ is bounded for $p>2$ in the case when $n=2$.

For $n\geq 2$, the lacunary spherical maximal operator (of order $0$) defined by Coifman-Weiss (\cite{CoifmanWeiss1978}) as 
\bes 
M_{\mathrm{lac}}f(x)
=
\sup_{j\in\mathbb{Z}}
\left|
\int_{S^{n-1}}
f(x-2^j\omega)\,d\sigma(\omega)
\right|.\ees
The following theorem was proved by Coifman-Weiss (\cite{CoifmanWeiss1978}) and independently by Calder\'{o}n (\cite{calderon1979}).
\begin{theorem}[Coifman--Weiss; Calder\'{o}n]
For every $1<p\leq\infty$, there exists
a constant $C_{n,p}>0$ such that
\[
\|M_{\mathrm{lac}}f\|_{L^p(\mathbb{R}^n)}
\leq
C_{n,p}\|f\|_{L^p(\mathbb{R}^n)}
\]
for all $f\in L^p(\mathbb{R}^n)$.
\end{theorem}

  The introduction of spherical means of complex order provides a broader analytic framework in which the singularity of the averaging kernel can be varied continuously and analytic interpolation methods can be applied. For a nice survey of this topic on $\R^n$ we refer \cite{Roncal}.

An early fundamental contribution to maximal-function theory on noncompact symmetric spaces (arbitrary rank) is due to Str\"{o}mberg (\cite{Stromberg1981}), who established weak-type $L^1$ estimates for (centered) Hardy-Littlewood maximal functions in this setting. 
The spherical maximal function of order $\alpha$ was studied by El Kohen (\cite{ElKohen1980}) on real hyperbolic spaces. Let $\mathbb H^n$ denote the real hyperbolic space realized in the hyperboloid model. For $\operatorname{Re}\alpha>0$, El Kohen introduced the family of spherical operators

\be \label{Kohen-defn} M_t^\alpha f(x)=\frac{1}{\Gamma(\alpha)}
\frac{2e^t}{(e^t-1)^{2\alpha}(\sinh t)^{n-2}}
\int_{B(x,t)}
[e^t x-y]^{\alpha-1}f(y) \,dy,
\ee
where $B(x,t)$ denotes the geodesic ball in $\mathbb H^n$ centered at $x$ with radius $t$, and $[\cdot,\cdot]$ denotes the Minkowski bilinear form.

The corresponding spherical maximal operator of complex order $\alpha$ is defined by

\bes 
\mathfrak m^\alpha f(x)=\sup_{t>0}
\left|M_t^\alpha f(x)\right|.
\ees

The family $M_t^\alpha$ is initially defined for $\operatorname{Re}\alpha>0$ and admits an analytic continuation to
\bes
\operatorname{Re}\alpha>\frac{1-n}{2}.
\ees

For $\alpha=0$, the analytically continued operator coincides with the ordinary spherical mean,
\bes
f*d\sigma_t(x),
\ees
where $d\sigma_t$ is the normalized surface measure on the geodesic sphere of radius $t$. Consequently,

\bes \mathfrak m^0 f(x)=\sup_{t>0}
\left|f*d\sigma_t(x)\right|,
\ees
is the usual spherical maximal operator on $\mathbb H^n$. 
\begin{theorem}[El Kohen, \cite{ElKohen1980}] Let $n\geq 2$. Then \bes
\|\mathfrak m^\alpha f\|_{L^p(\mathbb H^n)}
\leq C\|f\|_{L^p(\mathbb H^n)},
\ees
provided
\bes
\operatorname{Re}\alpha> 1-n+\frac{n}{p},
\qquad 1<p\leq 2,
\ees
or
\bes
\operatorname{Re}\alpha > \frac{2-n}{p},
\qquad 2<p\leq\infty.
\ees
\end{theorem}

The ordinary spherical maximal problem (that is, for $\alpha=0$) on general rank-one noncompact symmetric spaces was subsequently investigated by Ionescu in connection with radial Fourier integral operators, and he proved that the operator is bounded on $L^p(\mathbb H^2)$ for $2<p\leq \infty$.

More recently, Chen, Shen, Wang and Yan (\cite{ChenShenWangYan2025}) returned to the maximal operators of complex order on real hyperbolic spaces. They obtained several necessary conditions for boundedness. Their results demonstrate that the admissible range is influenced by both local Euclidean-type obstructions and the global hyperbolic geometry.
 Their result can be stated as follows.
\begin{theorem}[Chen--Shen--Wang--Yan, \cite{ChenShenWangYan2025}]\label{CSWY}
Let $n\geq 2$, $\alpha\in\mathbb{R}$ and $1<p<\infty$. Suppose that the spherical maximal operator
$\mathfrak m^\alpha f$
is bounded on $L^p(\mathbb{H}^n)$. Then the following necessary conditions hold:

(i) If $1<p\leq 2$, then
\bes
\alpha>1-n+\frac{n}{p}.
\ees

(ii) If $2<p<\infty$, then
\bes
\alpha
\geq
\max\left\{
\frac{1}{p}-\frac{n-1}{2},
\frac{1-n}{p}
\right\}.
\ees
\end{theorem}
In a related direction, Wang and Zhang (\cite{WangZhang}) studied lacunary spherical maximal operators on real hyperbolic spaces and showed that the geometry at infinity produces a substantially different phenomenon from the corresponding Euclidean problem.

Let
\bes
   \Lambda = \{2^{-j} \, |\,  j\in \Z_+ \} \cup \Z_+. 
\ees
Consider the lacunary spherical maximal operator $\cL_*^\alpha$ (of order $\alpha$) on the real hyperbolic space $\mathbb H^n$, defined by
\be \label{defn: lacunary maximal op}
    \cL_*^\alpha (f)(x) = \sup_{t \in \Lambda} \left| M_t^\alpha f(x) \right|, 
\ee
for $x\in \mathbb H^n.$
They proved the following theorem: 

\begin{theorem}[Wang and Zhang]\label{thm: main thm-lacunary} Let $1<p\leq \infty$ and $\alpha \in \C$ with 
    \bes
        \operatorname{Re }\alpha > (n-1)\left|\ \frac{1}{p} - \frac{1}{2} \right|  - \frac{n-1}{2}.
        \ees
    Then the operator $\cL_*^\alpha$ is bounded on $L^p(\mathbb H^n)$.
\end{theorem}

Let $X=G/K$ be a rank-one Riemannian symmetric space of noncompact type and let $d\sigma_\tau$ denote the normalized spherical measure, defined by
\bes
\int_X f(x)d\sigma_\tau=\int_K f\bigl(ka_\tau\cdot o\bigr) \, dk.
\ees

For $T\geq 0$, define the maximal operator
\bes
\mathcal M_T f(x)=\sup_{\tau\in[T,T+1]}
\left|f*d\sigma_\tau(x)\right|.
\ees

Ionescu proved the following result.
\begin{theorem} [Ionescu, \cite{Ionescu2000}]
     Let $X$ be a noncompact symmetric space of real rank one and let $n=\dim X$. If
\bes
\frac{n}{n-1}<p<\infty
\ees
and $T\geq0$, then
\bes
\|\mathcal M_T f\|_{L^p(X)}
\leq
C_p
e^{-\rho(1-\gamma_p)T}
(T+1)^\beta
\|f\|_{L^p(X)},
\ees
where
\bes
\gamma_p=\left|\frac{2}{p}-1\right|.
\ees
The constant $\beta$ may be taken to be
\bes
\beta=1,\qquad n\geq3,
\ees
and
\bes
\beta=2,\qquad n=2.
\ees
\end{theorem}
In particular, since
\bes
1-\gamma_p>0
\ees
for $1<p<\infty$, the norm of $\mathcal M_T$ decays exponentially as $T\rightarrow\infty$.

As a consequence, Ionescu obtained the boundedness of the spherical maximal operator:
\bes
\left\|
\sup_{0\leq\tau<\infty}
|f*d\sigma_\tau|
\right\|_{L^p(X)}
\leq
C_p\|f\|_{L^p(X)}
\ees
whenever
\bes
\frac{n}{n-1}<p\leq\infty.
\ees

Moreover, the exponential part of the decay,
\bes
e^{-\rho(1-\gamma_p)T},
\ees
is sharp.

Our aim in this article is to introduce the spherical maximal function and the lacunary spherical maximal function of order $\mu$ on general rank-one Riemannian symmetric spaces of noncompact type and to study its $L^p$-boundedness properties. This class includes the real hyperbolic spaces. 
To define the spherical maximal function of order $\mu$, we observe that $[e^t x-y]$ in (\ref{Kohen-defn}) can be written as 
\bes 2e^t\bigl(\cosh t-\cosh d(x,y)\bigr).\ees
Then we define the maximal operator of order $\mu\in \mathbb C$ on $X$ as
\bes
\mathcal M^\mu f(x)
=\sup_{t>0}|\mathcal M_t^\mu f(x)|,
\ees
where $\mathcal M_t^\mu$ is defined by 
\bes 
\mathcal M_{t}^{\mu} f(x) = \frac{1}{\Gamma(\mu)}    C(t,\mu)   \int_{B(x, t)} f(y) \left( \cosh 2t - \cosh 2d(x, y) \right)^{\mu - 1} d\sigma(y),
\ees
where the constant $C(t, \mu)$ is given by
\bes
    C(t,\mu) = \frac{4 \, e^{2\mu t}}{\sinh^{2\alpha_0 + 2 \mu}(t) \, \cosh^{2 \beta_0 + 2\mu}(t)}.   
\ees
Let \bes
   \Lambda = \{2^{-j} \, |\,  j\in \Z_+ \} \cup \Z_+. 
\ees
For $\mu \in \C$, we define the lacunary spherical maximal operator as
\bes
    \cL_*^\mu (f)(x) = \sup_{t \in \Lambda} \left| \mathcal M_t^\mu (f)(x)  \right|.
\ees
Let $m_1$ and $m_2$ denote the multiplicities of the roots $\gamma$ and $2\gamma$, respectively, arising in the restricted root-space decomposition of the Lie algebra $\mathfrak g$ of $G$. In contrast with the real hyperbolic case, the non-real rank-one symmetric spaces may have $m_2\neq0$. This additional root multiplicity affects both the radial density and the parameters of the associated Jacobi functions.

As in the Euclidean and real hyperbolic settings, we shall see that when $\mu=0$, the analytically continued family reduces to the usual spherical mean operator, whereas for $\mu=1$ it is comparable to the centered Hardy-Littlewood maximal operator.

We now state our main theorem for rank-one symmetric spaces of noncompact type.
Let $d=\dim X=1+ m_1 + m_2$.
\begin{theorem}\label{main-thm}
  The spherical maximal function $\mathcal M^\mu$ of order $\mu$ is a bounded operator on $L^p(G/K)$ for:
  \begin{enumerate}[label=(\arabic*)]
        \item $1 < p \le 2,  \text{ if } \operatorname{Re}\,\mu > 1-d+ \frac{d}{p}$,
        \item $2 \le p \le \infty, \text{ if } \operatorname{Re}\,\mu > \frac{2-d}{p}.$
    \end{enumerate}
  
\end{theorem}

The idea of the proof follows the approach of Ionescu and makes use of the Harish--Chandra series expansion together with oscillatory integral estimates. For small radii, we apply Sogge's result on local maximal functions associated with hypersurfaces. For large radii, the argument relies on the exponential decay arising from spherical analysis on $G/K$. Thus, the proof reflects the two distinct geometric regimes of a noncompact rank-one symmetric space: the local Euclidean-type behaviour near the origin and the exponential behaviour at infinity.
We have the following necessary condition:
\begin{theorem}\label{thm:necessary}
    Suppose $\mu \in \R$ and $\mathcal M^\mu $ is bounded on $L^p(G/K)$ for $1 < p < \infty$. Then we must have $\mu > 1 - d + \frac{d}{p}$.
\end{theorem}

We also have the following theorem for the lacunary spherical maximal function of order $\mu$.
\begin{theorem}\label{thm: L_*mu is bdd op}
    The lacunary spherical maximal operator $\cL_*^\mu$ is a bounded operator on $L^p(G/K)$ for all $1<p\le \infty$ and for $\mu \in \C$ with 
    $$
        \operatorname{Re }\mu > (d-1)\left|\ \frac{1}{p} - \frac{1}{2} \right|  - \frac{d-1}{2}.
    $$
   \end{theorem}
 We observe that in the case of the lacunary spherical maximal function of order $\mu$, we get a larger range of $\mu$ than the corresponding case of the spherical maximal function.

Section 2 provides the necessary preliminaries. In Sections 3 and 4, we prove the $L^p$-boundedness ($1 < p \le \infty$) of the spherical maximal operator and the lacunary spherical maximal operator, respectively.

\section{Preliminaries}
The letters $\mathbb N$, $\mathbb Z$, $\R$, and $\C$ will respectively denote the set of all natural numbers, the ring of integers, and the fields of real and complex numbers. For $ z \in \C $, we use the notation $\operatorname{Re} z$  for the real part of $z$. We shall follow the standard practice of using the letters $C$, $C_1$, $C_2$, etc., for positive constants whose values may change from one line to another. 

   \subsection{Symmetric spaces} We review some general facts about the necessary preliminaries regarding semisimple Lie groups and harmonic analysis on Riemannian symmetric spaces. Most of these are already known and can be found, for example, in \cite{GV, Helgason1984}. To make the article self-contained, we shall gather only those results  used throughout this paper.

Let $ G $ be a noncompact connected semisimple real rank-one Lie group with finite center, with its Lie algebra $\mathfrak g$. Let $ \theta $ be a Cartan involution of $ \mathfrak{g} $ and $ \mathfrak{g} =\mathfrak{k+p} $ be the associated Cartan decomposition. Let $ K=\exp \mathfrak{k} $ be a maximal compact subgroup of $ G $ and let $ X =G/K $ be an associated symmetric space with origin $ \textbf{0} =\{K\} $.  Let $ \mathfrak{a} $ be a maximal abelian subspace of $ \mathfrak{p} $. Since the group $ G $ is of real rank one,  dim $\mathfrak{a} =1 $.  Let $ \Sigma $ be the set of nonzero roots of the pair $ (\mathfrak{g,a}) $, and let $ W $ be the associated Weyl group. For the rank-one case, it is well known that  either $ \Sigma =\{-\gamma,\gamma\} $ or $ \{-2\gamma,-\gamma,\gamma,2\gamma \} $, where $ \gamma $ is a positive root and the Weyl group $ W $ associated to $ \Sigma  $ is  \{-Id, Id\}, where Id is the identity operator. Let $ \mathfrak{a}^+ =\{ H \in \mathfrak{a} : \alpha(H)>0 \} $ be a positive Weyl chamber, and let $ \Sigma^+ $ be the corresponding set of positive roots.  In our case, $ \Sigma ^+ = \{\alpha\}$ or $ \{ \alpha, 2\alpha\} $. For any root $ \beta \in \Sigma  $, let $ \mathfrak{g}_\beta $ be the root space associated to $ \beta $. Let 
			\bes 
			\mathfrak{n} =\sum_{\beta \in \Sigma^+} \mathfrak{g}_\beta.\ees  Also let 
			\bes  N =\exp  \mathfrak{n}.
			\ees
The group $ G $ has an Iwasawa decomposition \bes G= K (\exp \mathfrak{a})N,
			\ees and a Cartan decomposition 
			\bes G=K(\exp \mathfrak{a}^+) K.\ees This decomposition is unique. For each $ g \in G $, we denote $ H(g) \in \mathfrak{a} $
			and $ g^+ \in \ol {\mf{a}^+} $ are the unique elements such that
			\be\label{defn:Hg}
			g=k \exp H(g) n,\quad k\in K, n\in N,
			\ee  and 
\be\label{defn:gplus}
g=k_1 \exp (g^+) k_2,\quad k_1, k_2\in K.
\ee		

Let  $H_0$ be the unique element in $\mathfrak a$ such that $\alpha(H_0)=1$ and through this we identify $\mathfrak a$ with $\R$ as $t\leftrightarrow tH_0$ and $\mathfrak a_+= \{H\in \mathfrak a\mid \alpha(H)>0\}$ is identified with the set of positive real numbers.   We also identify $\mathfrak a^*$ and its complexification $\mathfrak a^*_\C$ with $\R$  and $\C$ respectively by $t\leftrightarrow t\alpha$ and  $z\leftrightarrow z\alpha$, $t\in \R$, $z\in \C$. Let $A=\exp\mathfrak{a}=\left\{a_t:=\exp(t H_0)\mid t\in\R\right\}$ and $A^+=\left\{a_t\mid t>0\right\}$.
 Let $m_1=\dim \mathfrak g_\gamma$ and $m_{2}=\dim \mathfrak g_{2\gamma}$ where $\mathfrak g_\gamma$ and
 $\mathfrak g_{2\gamma}$ are the root spaces corresponding to $\gamma$ and $2\gamma$ respectively. As usual, then $\rho=\frac 12(m_1+2m_2)$ denotes the half sum of the positive roots.
By abuse of notation we will denote $\rho(H_0)=\frac 12(m_1+2m_2)$ by $\rho$.

Let $ dg, dk$ and $dn$ be the Haar measures on the groups $ G, K$ and  $N$ respectively. We  normalize $ dk $ such that $ \int_K dk =1 $.  We have the following integral formulae corresponding to the Iwasawa and Cartan
decomposition respectively, which holds for any integrable function $f$:

\begin{equation}\label{eqn:integral-decom-nbar-a}
				\int_{G}f(g) dg =C\int_{N} \int_{\R} \int_{K} f(k a_t n)e^{2\rho t} \,dk \,dt\, dn,
			\end{equation}
	and		
			\begin{equation}
\int_Gf(g)dg=C\int_K\int_{\R^+}\int_K f(k_1a_tk_2) \Delta(t)\,dk_1\,dt\,dk_2. \label{polar}
\end{equation} 
where $\Delta(t)=(\sinh t)^{m_1 + m_2}(\cosh t)^{m_2}$.

\subsection{Fourier transform}For a sufficiently nice function $f$ on $X$, its Helgason Fourier transform $\widetilde{f}$ is a function defined on $\C \times K$ given by 
\be \label{defn:hft}
\widetilde{f}(\lambda,k) = \int_{G} f(g) e^{(i\lambda- \rho)H(g^{-1}k)}\, dg, \quad \lambda \in \C, k \in K, 
\ee
whenever the integral exists (\cite[p. 199]{Helgason1984}). 

It is known that if $f\in L^1(X)$ then $\widetilde{f}(\lambda, k)$ is a continuous function of $\lambda \in \R$, for almost every $k\in K$. If in addition $\widetilde{f}\in L^1(\R\times K, |c(\lambda)|^{-2}~d\lambda~dk)$, then the following Fourier inversion holds,
\be\label{hft}
f(gK)= |W|^{-1}\int_{\R\times K}\widetilde{f}(\lambda, k)~e^{-(i\lambda+\rho)H(g^{-1}k)} ~ |c(\lambda)|^{-2}d\lambda~dk,
\ee
for almost every $gK\in X$ (\cite[Chapter III, Theorem 1.8, Theorem 1.9]{Helgason1984}), where $c(\lambda)$ is the Harish Chandra's $c$-function given by \bes c(\lambda)=\frac{2^{\rho-i\lambda} \Gamma(\frac{m_1 + m_2 +1}{2})\Gamma(i\lambda)}{\Gamma(\frac{\rho+ i\lambda}{2}) \Gamma(\frac{m_1+ 2}{4} + \frac{i\lambda}{2})}.\ees
It is normalized such that $c(-i\rho)=1$.

 Moreover, $f \mapsto \widetilde{f}$ extends to an isometry of $L^2(X)$ onto $L^2(\R\times K, |c(\lambda)|^{-2}~d\lambda~dk )$ (\cite[Chapter III, Theorem 1.5]{Helgason1984}).

A function $f$ is called $K$-biinvariant if \bes f(k_1xk_2)=f(x) \text{ for all } x\in G, k_1, k_2\in K.\ees  We denote the set of all $K$-biinvariant functions by $\mathcal F (G//K)$.
Let $\mathbb D(G/K)$ be the algebra of $G$-invariant differential operators on $X$. The elementary spherical functions $\phi$ are $C^\infty$ functions and are joint eigenfunctions of all $D\in\mathbb D(G/K)$ for some complex eigenvalue $\lambda(D)$. That is $$D\phi=\lambda(D)\phi, \quad D\in\mathbb D(G/K).$$They are parametrized by $\lambda\in\C$. The algebra $\mathbb D(G/K)$ is generated by the Laplace-Beltrami operator $ L$. Then we have, for all $\lambda\in\C, \varphi_\lambda$ is a $C^\infty$ solution of  
\begin{equation}\label{phi-lambda}
L\phi=-(\lambda^2 + \rho^2)\phi.
\end{equation}

For any $\lambda\in \C$, the elementary spherical function $\varphi_\lambda$ has the following integral representation
$$\varphi_\lambda(x)=\int_K e^{-(i\lambda+\rho)H(xk)}\,dk \text{ for all } x\in G.$$
The spherical transform $\mathcal{F} = \what{f}$ of a  suitable $K$-biinvariant function $f$ is defined by the formula:
$$\mathcal{F}(\lambda) = \what{f}(\lambda)=\int_Gf(x)\varphi_\lambda(x^{-1})\,dx.$$
It is easy to check that, for a suitable $K$-biinvariant function $f$ on $G$; its Helgason Fourier transform $\widetilde{f}$ reduces to the spherical transform $\what{f}$.
For a suitable function $f$ on $X$ and a $K$-invariant function $g$, we have
\bes
\widetilde{f\ast g}(\lambda, k)=\widetilde{f}(\lambda, k)\widehat{g}(\lambda).
\ees
We now list down some well-known properties of the elementary spherical functions which are important for us (\cite[ Prop 3.1.4 and Chapter 4, \S 4.6]{GV}, \cite[Lemma 1.18, p. 221]{Helgason1984}).

\begin{enumerate}
\item[(1)] $\varphi_\lambda(g)$ is $K$-biinvariant in $g\in G$,  $\varphi_\lambda=\varphi_{-\lambda}$, $\varphi_\lambda(g)=\varphi_\lambda(g^{-1})$.
\item[(2)] $\varphi_\lambda(g)$ is $C^\infty$ in $g\in G$ and holomorphic in $\lambda\in\C$.
\item[(3)] The following inequality holds:
\bes
e^{-\rho t} \leq \varphi_0(a_t)\leq \left(1+|t|\right)~e^{-\rho t}, \, t\geq 0.
\ees
\item[(4)] $|\varphi_\lambda(x)|\leq 1$ for all $x\in G$ if and only if $\lambda\in S_1=\left\{\lambda\in \C \mid |\Im\lambda|\leq \rho\right\}$.
\item[(5)] For all $\lambda\in \R$ we have
\bes
|\varphi_\lambda(g)| \leq  \varphi_0(g)\leq 1.
\ees
\end{enumerate}

A Jacobi function (see \cite{Koornwinder1984}) $\phi_\lambda^{(\alpha, \beta)} (\alpha, \beta, \lambda\in\C, \alpha\not=-1, -2, \cdots)$ is defined as the even $C^\infty$ function on $\R$ such that $\phi_\lambda^{(\alpha, \beta)}(0)=1$, and it satisfies the following differential equation 
\begin{equation}\label{eqn-1}
\left(\frac{d^2}{dt^2} + ((2\alpha +1)\coth t + (2\beta +1)\tanh t)\frac{d}{dt} + \lambda^2 + (\alpha +\beta +1)^2\right)\phi_\lambda^{(\alpha, \beta)}(t)=0.
\end{equation}
 This Jacobi function can be written as the hypergeometric function:

\begin{equation}
\phi_\lambda^{(\alpha, \beta)}(t)= {}_2F_1\left(\frac{\alpha +\beta + 1-i\lambda}{2}, \frac{\alpha + \beta + 1+ i\lambda}{2}; \alpha +1; -\sinh^2t\right).
\end{equation}

The $A$-radial part of the Laplace-Beltrami operator $L$ on $X$ is given by 
\begin{equation} 
L_Af(a_t):=\frac{d^2}{dt^2}f (a_t)+\left((m_1+ m_2)\coth t + m_2\tanh t\right)\frac{d}{dt}f(a_t), t>0. 
\end{equation}
Therefore equation (\ref{phi-lambda}) reduces to 
\begin{equation}\label{diff-hyper}
\frac{d^2\phi}{dt^2}+\left((m_1+ m_2)\coth t + m_2\tanh t\right)\frac{d\phi}{dt} + (\lambda^2 + \rho^2)\phi=0, t>0.
\end{equation}

Therefore the spherical function can be given in terms of Jacobi-function as, \begin{eqnarray}\label{phi}\varphi_\lambda(a_t)=\phi^{(\alpha_0, \beta_0)}_\lambda(t)= {}_2F_1\left(\frac{\rho-i\lambda}{2}, \frac{\rho+ i\lambda}{2}; \frac{m_1 + m_2 +1}{2}; -\sinh^2t\right),
\end{eqnarray}
where \be
\label{alpha_0} \alpha_0=\frac{m_1 + m_2-1}{2} \text{ and } \beta_0=\frac{m_2-1}{2}.\ee

Let \bes
\Delta_{\alpha, \beta}(s) = \sinh^{2\alpha + 1}(s) \cosh^{2\beta + 1}(s).
\ees

The Harish-Chandra $c$-function (associated with parameter $(\alpha, \beta)$) is given by (\cite{Koornwinder1984}, 2.18)
\bes
 c_{\alpha,\beta}(\lambda) := 
    \dfrac{2^{\rho - i\lambda} \, \Gamma(\alpha + 1) \, \Gamma(i \lambda)}
    {\Gamma \left(\dfrac{i\lambda + \rho}{2} \right) \,\Gamma \left(\dfrac{i\lambda + \alpha - \beta +1}{2} \right)
    }.
\ees
We need the estimates for the Harish-Chandra $c$-functions. 
\begin{lemma}[\cite{Ionescu2000}, Prop. A1]\label{lemma:estimate for lambda^-1 c(lambda)^-1}
    Suppose $\lambda \in \R$ and $N\in \mathbb N$. The Harish-Chandra function $c(\lambda)$ satisfies 
    \bes
        |c(\lambda)|^{-2} = (c(\lambda) \, c(-\lambda))^{-1}.
    \ees    
    The function $\lambda^{-1}c(-\lambda)^{-1}$ belongs to $C^\infty(\R)$ and
    \bes
        \left |  \dfrac{\partial^k}{\partial\lambda^k} \left ( \lambda^{-1}c(-\lambda)^{-1} \right ) \right | \le C_k \, (1 + |\lambda|)^{\alpha_0 - 1/2 - k}.
    \ees
    for all integers $k \in [0,N]$.
\end{lemma}

By analogy with $c(\lambda)$, we define another function 
\bes
    c^\mu(\lambda) := c_{\alpha_0+\mu,\beta_0+\mu}(\lambda) = \dfrac{2^{\rho+2\mu - i\lambda} \, \Gamma(\alpha_0 + \mu + 1) \, \Gamma(i \lambda)}
    {\Gamma \left(\dfrac{i\lambda + \rho}{2} + \mu \right) \,\Gamma \left(\dfrac{i\lambda + \alpha_0 - \beta_0 +1}{2} \right)
    } \quad \text{for Re }\mu > -\alpha_0 -1,
\ees
where $\alpha_0, \beta_0$ is defined in (\ref{alpha_0}).
We observe that $c^0=c$.
For the function $c^\mu$, we have the following result.
\begin{lemma}\label{lemma:estimate for lambda c^mu(lambda)}
    Let $ \operatorname{Re }\mu > -\alpha_0 - 1$ and $k\in \mathbb{N}$. The function $\lambda\, c^\mu(\pm\lambda)$ belongs to $C^\infty(\R)$ and
    \bes
        \left |  \dfrac{\partial^k}{\partial\lambda^k} \left ( \lambda\, c^\mu(\pm\lambda) \right ) \right | \le C_k \, (1 + |\lambda|)^{1/2 - \alpha_0 - \operatorname{Re }\mu - k}.
    \ees
\end{lemma}
\begin{proof}
    This is a consequence of Stirling's formula (\cite{titchmarsh1939}, Chap. 4).
\end{proof}

\begin{lemma}\label{lemma:estimate for c^mu(lambda)/c(lambda)}
    Let $ \operatorname{Re }\mu > -\alpha_0 - 1$ and $k\in \mathbb{N}$. The function $\frac{c^\mu(\lambda)}{c(\pm\lambda)}$ belongs to $C^\infty(\R)$ and
    \bes
        \left |  \dfrac{\partial^k}{\partial\lambda^k} \left( \frac{c^\mu(\lambda)}{c(\pm\lambda)} \right) \right | \le C_k \, (1 + |\lambda|)^{ - \operatorname{Re }\mu - k}.
    \ees
\end{lemma}
\begin{proof}
    This is a consequence of Stirling's formula (\cite{titchmarsh1939}, Chap. 4).
\end{proof}

We also need the following Herz's majorization principle (\cite{Cowling1997}, \cite[Proposition 2.1]{Ionescu2003}):
\begin{lemma} \label{lemma: Kunze-Stein phenomenon}
    Suppose that $1<p<2$ and $k$ is a non-negative radial function on $G/K$. Then we have, for all $f\in L^p(G/K)$, 
    \bes
        \|f*k\|_p \le C\, \left( \int_0^\infty e^{-\frac{2\rho \, r}{p'}} \, k(r) \, \sinh^{2\alpha_0+1} r \, \cosh^{2\beta_0+1} r \, dr \right) \|f\|_p.
    \ees
\end{lemma}
\section{spherical maximal function}
We first recall that 
\bes
\alpha_0=\frac{m_1 + m_2-1}{2} \text{ and } \beta_0=\frac{m_2-1}{2}.\ees 
Then we consider the following analytic family of operators
\bes 
\mathcal M_{t}^{\mu} f(x) = \frac{1}{\Gamma(\mu)}    C(t,\mu)   \int_{B(x, t)} f(y) \left( \cosh 2t - \cosh 2d(x, y) \right)^{\mu - 1} dy,
\ees
where,
\be \label{eqn: c(t,mu)}
    C(t,\mu) = \frac{4 \, e^{2\mu t}}{\sinh^{2\alpha_0 + 2 \mu}(t) \, \cosh^{2 \beta_0 + 2\mu}(t)},   
\ee
and 
\bes
\mathcal M^\mu f(x)
=\sup_{t>0}|\mathcal M_{t}^\mu f(x)|.
\ees

We first show that this spherical maximal function turns out to be the usual spherical maximal function for the case when $\mu=0$, if $m_1 + m_2>1$. 
If $y = k a_s \cdot 1$, then \bes \cosh 2d(1, y) = \cosh 2s.\ees
Therefore, 
\beas
      \mathcal M_{t}^{\mu} f(1) &=& C(t, \mu)  \frac{1}{\Gamma(\mu)} \int_0^t \int_K f(k a_s \cdot 1) \left( \cosh 2t - \cosh 2s \right)^{\mu - 1}\sinh^{2\alpha_0 +1}(s) \cosh^{2\beta_0+1}(s) \, ds \, dk \\ \\
    &=& C(t, \mu)    \frac{1}{\Gamma(\mu)} \int_0^t f^{\sharp}(s) \left( \cosh 2t - \cosh 2s \right)^{\mu - 1} \Delta_{\alpha_0,\beta_0}(s) \, ds.
 \eeas

where, 
\bes
    f^{\sharp}(s) = \int_K f(k a_s \cdot 1) \, dk \text{ and }
    \Delta_{\alpha_0,\beta_0}(s) = \sinh^{2\alpha_0 +1}(s) \cosh^{2\beta_0+1}(s).
\ees

We first observe that
\beas
   & \frac{1}{\Gamma(\mu)} \int_0^t f^\sharp(s) \left( \cosh 2t - \cosh 2s \right)^{\mu - 1} \Delta_{\alpha_0,\beta_0}(s) \, ds &
    \\
    &= - \frac{1}{\mu \, \Gamma(\mu)} \int_0^t f^\sharp(s) \frac{d}{ds} \left( \left( \cosh 2t - \cosh 2s \right)^\mu \right) \frac{\Delta_{\alpha_0,\beta_0}(s) \, ds}{2 \sinh 2s}.& 
    \eeas
    The last term is equal to
     \begin{multline*}
       - \frac{1}{\Gamma(\mu + 1)} \left[ f^\sharp(s)    \frac{\Delta_{\alpha_0,\beta_0}(s)}{2 \sinh 2s}    \left( \cosh 2t - \cosh 2s \right)^\mu \right]_{s=0}^t \\
     + \frac{1}{\Gamma(\mu + 1)}    \int_0^t \frac{d}{ds} \left( f^\sharp(s)    \frac{\Delta_{\alpha_0,\beta_0}(s)}{2 \sinh 2s} \right) \left( \cosh 2t - \cosh 2s \right)^\mu ds,   
     \end{multline*}
     
 which is equal to    
      \bes 
      \frac{1}{\Gamma(\mu + 1)} \int_0^t \frac{d}{ds} \left( f^\sharp(s) \frac{\Delta_{\alpha_0,\beta_0}(s)}{2 \sinh 2s} \right) \left( \cosh 2t - \cosh 2s \right)^\mu ds,
\ees
since $\alpha_0 >0$.
Observe that the following term
\bes \int_0^t \frac{d}{ds} \left( f^\sharp(s) \frac{\Delta_{\alpha_0,\beta_0}(s)}{2 \sinh 2s} \right) \left( \cosh 2t - \cosh 2s \right)^\mu ds,
\ees
converges for $\operatorname{Re}\mu > -1$. Therefore we have,
\bes
    \mathcal M_{t}^\mu f(1) = C(t, \mu) \frac{1}{4 \Gamma(\mu + 1)} \int_0^t \frac{d}{ds} \left( f^\sharp(s) \sinh^{2\alpha_0}(s) \cosh^{2\beta_0}(s) \right) \left( \cosh 2t - \cosh 2s \right)^\mu ds.
\ees

Thus, 
\bes
    \left. \mathcal M_{t}^{\mu} f(1) \right|_{\mu = 0} =  \frac{C(t, 0)}{4 \Gamma(1)} \int_0^t \frac{d}{ds} \left( f^{\sharp}(s) \sinh^{2\alpha_0}(s) \cosh^{2\beta_0}(s) \right) ds
    = \int_K f(k a_{t} \cdot 1) \, dk.
\ees

Using translation by the elements of the group $G$, we have
\bes
\mathcal M_{t}^{0} f(x) = \int_K f(u k a_{t} 1) \, dk,
\ees
where $x = u\cdot 1$, which
is the spherical mean on the rank-one symmetric spaces.

It is also easy to see that the spherical maximal function $\mathcal M^\mu$ becomes comparable to the centered Hardy-Littlewood maximal function for the case when $\mu=1$.

Also the spherical function $\varphi_\lambda$ is equal to the Jacobi function $\phi_\lambda^{(\alpha_0, \beta_0)}$ (see (\ref{phi})), where $\alpha_0=\frac{m_1 + m_2-1}{2}$ and  $\beta_0=\frac{m_2-1}{2}$.

\begin{lemma}
   The operator $\mathcal M_t^\mu$ is a multiplier operator and for $\operatorname{Re}\mu>0$, the multiplier is given in terms of the Jacobi function. More precisely, for $\operatorname{Re}\mu>0$, \bes
    m_t^\mu(\lambda) =  2^\mu \, \frac{e^{2\mu t} \,\Gamma(\alpha_0+1)}{\Gamma(\alpha_0+\mu+1)} \, \phi_\lambda^{(\alpha_0+\mu, \, \beta_0+\mu)}(t).
\ees
\end{lemma} 
\begin{proof}

We write,
\beas
    \mathcal M_{t}^{\mu} f(x) &=& \frac{1}{\Gamma(\mu)} C(t, \mu) \int_{B(x, t)} f(y) \left( \cosh 2t - \cosh 2d(x, y) \right)^{\mu - 1} \, d\sigma(y)\\
   & =& \int_{G/K} \mathcal{K}'(x,y) f(y) \, dy,
\eeas
where,
\bes
    \mathcal{K}'(x, y) = \frac{1}{\Gamma(\mu)} C(t, \mu) \, \mathbbm{1}_{B(x, t)}(y) \left( \cosh 2t - \cosh 2 d(x, y) \right)^{\mu-1}.
\ees
Clearly, $\mathcal M_{t}^\mu$ commutes with the action of $G$. To see this, we need to show that 
\bes
    \mathcal M_{t}^\mu (L_g f)(x) = L_g (\mathcal M_{t}^\mu f)(x) \quad \text{ for all } x \in G/K \text{ and } g \in G
\ees
where, $L_g f(x) = f(g^{-1}x)$ for all $x\in G/K$ and $g\in G$.
Then
\bes
    \mathcal M_{t}^\mu (L_g f)(x)
    = \frac{C(t,\mu)}{\Gamma(\mu)} \int_{B(x, t)} f(g^{-1}y) \left( \cosh 2t - \cosh 2 d(x,y) \right)^{\mu-1} \,  d\sigma(y).
\ees
Using the change of variable $z = g^{-1}y $, we get
\bes
    \mathcal M_{t}^\mu (L_g f)(x)
    = \frac{C(t,\mu)}{\Gamma(\mu)} \int_{B(g^{-1}x, t)} f(z) \left( \cosh 2t - \cosh 2 d(g^{-1}x,z) \right)^{\mu-1} \, d\sigma(z) 
    = L_g \left( \mathcal M_{t}^\mu f \right)(x).
\ees
As $\mathcal M_{t}^\mu$ commutes with the action of $G$, it can be shown that for all  $g \in G$, \, \bes \mathcal{K}'(gx, gy) = \mathcal{K}'(x, y),\ees for a.e. $x, y \in G/K.$ To see this
\beas
    \mathcal M_{t}^\mu (L_g f)(x) &=& \int_{G/K} \mathcal{K}'(x, y) \, f(g^{-1}y) \, dy \\
    &=& \int_{G/K} \mathcal{K}'(x, gy) \, f(y) \, dy
\eeas
Also,
\bes
    L_g \left( \mathcal M_{t}^\mu f \right)(x) = \int_{G/K} \mathcal{K}'(g^{-1}x, y) \, f(y) \, dy
\ees
So we have  \bes \mathcal{K}'(g^{-1}x, y) = \mathcal{K}'(x, gy),\ees for a.e. $x, y \in G/K$. Therefore, \bes \mathcal{K}'(gx, gy) = \mathcal{K}'\left(g^{-1}(gx), y\right) = \mathcal{K}'(x, y),\ees for almost every  $x, y \in G/K.$ Now we will show that $\mathcal M_{t}^\mu$ is a Fourier multiplier operator. Consider \bes \mathcal{K}(w) = \mathcal{K}'(w \cdot 1, 1),\ees for $w\in G$. It is easy to see that $\mathcal{K}$ is bi-invariant. Let $x,y\in G/K$. Then, $x = v\cdot 1, \, y = u\cdot 1$ for some $u,v \in G$. Now 
\beas
    \mathcal M_{t}^\mu f(v) &=& \int_G \mathcal{K}'(v \cdot 1, u \cdot 1) \, f(u) \, du \\
    &=& \int_G \mathcal{K}'(u^{-1} v \cdot 1, 1) \, f(u) \, du \\
    &=& \int_G \mathcal{K}(u^{-1} v) \, f(u) \, du \\
    &=& \mathcal{K} * f(v).
    \eeas
    
Now
\beas
    \widehat{\mathcal{K}}(\lambda) &=& \int_G \mathcal{K}(u) \, \varphi_\lambda(u^{-1}) \, du \\
    &=& \int_G \mathcal{K}'(u \cdot 1, 1) \, \varphi_\lambda(u^{-1}) \, du \\
    &=& \int_G \mathcal{K}'(1, u^{-1} \cdot 1) \, \varphi_\lambda(u^{-1}) \, du \\
    &=& \int_G \mathcal{K}'(1, u \cdot 1) \, \varphi_\lambda(u) \, du \\
    &=& \frac{C(t, \mu)}{\Gamma(\mu)}  \int_G \mathbbm{1}_{B(1, t)}(u \cdot 1) \left( \cosh 2t - \cosh 2 d(1, u \cdot 1) \right)^{\mu-1} \varphi_\lambda(u)\, du \\
    &=& \frac{C(t, \mu)}{\Gamma(\mu)}  \int_0^t (\cosh 2t - \cosh 2s)^{\mu-1} \, \varphi_\lambda(s) \, \Delta_{\alpha_0, \beta_0}(s) \, ds,
\eeas
using  polar decomposition. Thus, we obtain
\bes
    \mathcal M_{t}^\mu f(x) = \mathcal K\ast f(x),
\ees
where,
\bes
    \widehat{\mathcal{K}}(\lambda)= m_{t}^\mu(\lambda) =\frac{1}{\Gamma(\mu)} C(t, \mu) \int_0^t \left( \cosh 2t - \cosh 2s \right)^{\mu - 1} \, \varphi_\lambda(s) \, \Delta_{\alpha_0, \beta_0}(s) \, ds.
\ees

For $\operatorname{Re}\mu>0$ and  $\operatorname{Re}\alpha>-1$ we have \cite[(2.14)]{Koornwinder1975},
\bes
    \int_0^t \left( \cosh 2t - \cosh 2s \right)^{\mu - 1} \phi_\lambda^{(\alpha, \beta)}(s) \, \Delta_{\alpha, \beta}(s) \, ds =
    \frac{2^{\mu-1} \Gamma(\alpha+1) \,  \Gamma(\mu) }{\Gamma(\alpha +\mu+1)  \,  \sinh 2t} \,  \Delta_{\alpha+\mu, \, \beta+\mu}(t) \, \phi_\lambda^{(\alpha+\mu, \, \beta+\mu)}(t),
    \ees 
    which is equal to 
   
\bes
    \frac{2^{\mu-2}\,\Gamma(\alpha+1) \, \Gamma(\mu)}{ \Gamma(\alpha +\mu+1)} (\sinh t)^{2\alpha+2\mu} \, (\cosh t)^{2\beta+2\mu} \, \phi_\lambda^{(\alpha+\mu, \, \beta+\mu)}(t).
\ees
Therefore since \bes
\varphi_\lambda=\phi_\lambda^{(\alpha_0, \beta_0)},
\ees  we have for $\operatorname{Re}\mu>0$,
\beas
    m_t^\mu(\lambda) &=& \frac{1}{\Gamma(\mu)}\, C(t, \mu) \, \frac{2^{\mu-2} \, \Gamma(\alpha_0+1) \, \Gamma(\mu)}{\Gamma(\alpha_0+\mu+1)} \, (\sinh t)^{2\alpha_0+2\mu} \, (\cosh t)^{2\beta_0+2\mu} \, \phi_\lambda^{(\alpha_0+\mu, \, \beta_0+\mu)}(t)\\ 
    &=&  2^\mu \, \frac{e^{2\mu t}\, \Gamma(\alpha_0+1)}{\Gamma(\alpha_0+\mu+1)} \, \phi_\lambda^{(\alpha_0+\mu, \, \beta_0+\mu)}(t).
\eeas
\end{proof}
Now, using analytic continuation we get for $\operatorname{Re}\mu>-\alpha_0-1$,
\bes
    m_t^\mu(\lambda) =  2^\mu \, \frac{e^{2\mu t} \,\Gamma(\alpha_0+1)}{\Gamma(\alpha_0+\mu+1)} \, \phi_\lambda^{(\alpha_0+\mu, \, \beta_0+\mu)}(t).
\ees
We have the following estimate of the multiplier function $m_t^\mu(\lambda)$ for $\operatorname{Re }\mu > -\alpha_0 - 1/2$.  The proof of the following lemma is similar (to the case when $\mu$=0, \cite{Ionescu2000}, Prop. A2). 
\begin{lemma}\label{lemma:estimate for m_t^mu(lambda)}
   \begin{enumerate}[label=(\roman*)] 
       \item Let $\operatorname{Re }\mu > -\alpha_0 - 1/2$ and  $\lambda \in \R$. If $t> 0 $, then
       \bes
            |m_t^\mu(\lambda)| \leq C_\mu \, (1+t) \, e^{-\rho t},
       \ees
       and
       \bes
            |\partial_t(m_t^\mu(\lambda))| \le C_\mu \, (1+|\lambda|)\, (1+t) \, e^{-\rho t}.
       \ees
       \item  Let $\operatorname{Re }\mu > -\alpha_0 - 1/2$ and let $t_0>0$ be fixed. Suppose $\lambda \in \R\setminus\{0\}$ and $N\in \mathbb N$. If $t \ge t_0$, then $m_t^\mu(\lambda)$ can be written in the form
       \bes
            m_t^\mu(\lambda) = e^{-\rho t} \left(  e^{i\lambda t} \, c^\mu(\lambda) \, a_2^\mu(\lambda, t) 
            + e^{-i\lambda t}\, c^\mu(-\lambda)\, a_2^\mu(-\lambda, t)
            \right),
       \ees
        where the function $a_2^\mu(\lambda,t)$ satisfies the following inequalities
        \bes
            \left |  \dfrac{\partial^k}{\partial\lambda^k} \,
            \dfrac{\partial^l}{\partial t^l}
            \left ( a_2^\mu(\lambda , t) \right ) \right | \le C_{k,l} \, (1 + |\lambda|)^{- k}.
        \ees
        for all integers $k\in [0,N]$, $l\in \{0,1\}$ and for all $t\ge t_0$. 
        \item Let $\operatorname{Re }\mu > -\alpha_0 - 1/2$ and  $\lambda \in \R$. If $0<t\le 1$, we have
        \bes
            |\partial_\lambda(m_t^\mu(\lambda))| \le C\, t.
        \ees
        Furthermore, if $|\lambda|\, t \ge 1$, then
        \bes
            |m_t^\mu(\lambda)| \le C \, (|\lambda|\,t)^{-(\operatorname{Re }\mu + \alpha_0+ \frac{1}{2})}.
        \ees
   \end{enumerate}
\end{lemma}
For $\mu =0$, we shall use the notation $a_2$ for the function $a_2^0$ in the rest of the paper. Let $t_0>0$ be fixed and assume $T\ge t_0$. Let $\psi_T: \R \to [0,1]$ be a smooth cutoff function such that
\begin{equation}  \nonumber
\psi_T = \left\{\begin{array}{lll}
1 & \text{ on } & [-1/2, T+3/2] \\
0 & \text{ outside }&  (-1, T+2).
\end{array}\right.
\end{equation}
Now, consider $t \in [T, T+1]$. 
We recall that
\bes
    \mathcal{K}(x) = \frac{1}{\Gamma(\mu)} C(t, \mu) \, \mathbbm{1}_{B(x, t)}(e) \left( \cosh 2t - \cosh 2 d(x, e) \right)^{\mu-1}.
\ees
We observe that, $\mathcal K$ is radial and the support of  $\mathcal{K}$ is contained in $[0, t]$. 
In terms of the Fourier inversion formula we have,
\bes
    \mathcal{K}(r)= \mathcal{F}^{-1}\left(m_t^\mu\right)(r)
    = \int_0^\infty m_t^\mu(\lambda) \, \phi_\lambda^{(\alpha_0, \beta_0)}(r) \, |c(\lambda)|^{-2} \, d\lambda.
\ees
Thus, for all $r$ in the support of $\mathcal{K}$, we have $\psi_T(r) \equiv 1$. Thus, whenever $r\in [0,t]$, we can write
\bes
    \mathcal{K}(r) = \psi_T(r) \, \int_0^\infty m_t^\mu(\lambda) \, \phi_\lambda^{(\alpha_0, \beta_0)}(r) \, |c(\lambda)|^{-2} \, d\lambda.
\ees
Let $\eta_0$ be an even, smooth cutoff function on $\R$ such that 
\begin{equation}  \nonumber
\eta_0(s) = \left\{\begin{array}{lll}
1 & \text{ if } & |s|\le 1 \\
0 & \text{ if } &  |s|\ge 2
\end{array}\right.
\end{equation}
and for $j=1,2,\cdots,$ let
\bes
\eta_j(s) = \eta_0 \left(\frac{s}{2^j}\right) - \eta_0 \left(\frac{s}{2^{j-1}}\right).
\ees
Clearly, supp $\eta_j \subset \{ s\in \R : |s|\in [2^{j-1}, 2^{j+1}] \}$ for any $j\ge 1$.
We define for any $r\in \R$,
\be 
    A_{t,j}^\mu(r) = \psi_T(r) \, \int_0^\infty \eta_j(\lambda) \, m_t^\mu(\lambda) \, \phi_\lambda^{(\alpha_0, \beta_0)}(r) \, |c(\lambda)|^{-2} \, d\lambda = \psi_T(r)\, \mathcal{F}^{-1}(\eta_j m_t^\mu)(r).
\ee
Then clearly, the support of $A_{t,j}^\mu$ is contained in $(-1,T+2)$. We also observe that $\mathcal{K}(r) = \sum_{j\ge 0} A_{t,j}^\mu(r)$ whenever $r\in [0,t]$ and for all $r$, $\mathcal{K}(r) \le \sum_{j\ge 0} |A_{t,j}^\mu(r)|$. We need the following lemma, whose proof is an adaptation of Ionescu's argument.

\begin{lemma}\label{lemma: sup_t[T,T+1] |f*A_t,j^mu|}
    Let $t_0> 0$ be fixed. Then for $T\ge t_0$, $j\ge 0$, we have
    \bes
        \left\| \, \sup_{t \in [T, T+1]} |f * A_{t,j}^\mu| \, \right\|_2 \le C \, 2^{-(\operatorname{Re } \mu + \alpha_0)\,j} \, e^{-\rho T} (1+T) \, \|f\|_2.
    \ees
\end{lemma}

\begin{proof}
We will prove the following two estimates
\be \label{lemma:estimate_1}
    \| f * A_{t,j}^\mu \|_2 \le C \, 2^{-(\operatorname{Re }\mu + \alpha_0 + \frac{1}{2})\,j} \, e^{-\rho T} \, (1+T) \, \|f\|_2
\ee
and
\be \label{lemma:estimate_2}
    \| \partial_t (f * A_{t,j}^\mu) \|_2 \le C \, 2^{-(\text{Re}\,\mu + \alpha_0 - \frac{1}{2})\,j} \, e^{-\rho T} \, (1+T) \, \|f\|_2,
\ee
for any $t\in [T,T+1]$. Let us first show (\ref{lemma:estimate_1}). To do this, let 
\bes
    B_{t,j}^\mu(r) = (1 - \psi_T(r)) \, \mathcal{F}^{-1} (\eta_j m_t^\mu)(r) 
\ees
be the complementary kernel of $A_{t,j}^\mu$, such that $A_{t,j}^\mu + B_{t,j}^\mu = \mathcal{F}^{-1} (\eta_j m_t^\mu)$. Then, using Plancherel's theorem, we have
\bes
    \|f * (A_{t,j}^\mu + B_{t,j}^\mu)\|_2 = \|\mathcal{F}(f) \, \mathcal{F}(A_{t,j}^\mu+ B_{t,j}^\mu) \|_2 \leq \|\eta_j m_t^\mu\|_\infty \, \|f\|_2
\ees
Using Lemma \ref{lemma:estimate for lambda c^mu(lambda)} and Lemma \ref{lemma:estimate for m_t^mu(lambda)}(ii), we get for $\lambda \in $ supp($\eta_j$) and for $j>0$,
\begin{align*}
    \left|m_t^\mu(\lambda) \right| &\leq e^{-\rho t}\, \frac{1}{|\lambda|} \, \left(  \left|\lambda\, c^\mu(\lambda) \right|\, \left|a_2^\mu(\lambda,t) \right| + \left|\lambda\, c^\mu(-\lambda) \right|\, \left|a_2^\mu(-\lambda,t)\right| \right)\\
    &\leq C \, 2^{-(\text{Re}\,\mu + \alpha_0 + \frac{1}{2})\,j} \, e^{-\rho t}.
\end{align*}
    
Hence for $j>0$,
\bes
     \|f * (A_{t,j}^\mu + B_{t,j}^\mu)\|_2 \leq C \, 2^{-(\text{Re}\,\mu + \alpha_0 + \frac{1}{2})\,j} \, e^{-\rho T} \, \|f\|_2.
\ees
For $j=0$, we will use Lemma \ref{lemma:estimate for m_t^mu(lambda)}(i) to obtain 
\bes
     \|f * (A_{t,0}^\mu + B_{t,0}^\mu)\|_2 \leq C \, e^{-\rho T} \,(1+T) \|f\|_2.
\ees
Thus, for any $t\in[T,T+1]$ and for $j\ge 0$, 
\be\label{eqn: f * (A_t,j^mu + B_t,j^mu)}
    \|f * (A_{t,j}^\mu + B_{t,j}^\mu)\|_2 \leq C \, 2^{-(\text{Re}\,\mu + \alpha_0 + \frac{1}{2})\,j} \, e^{-\rho T} \, (1+T) \, \|f\|_2.
\ee
On the other hand, we have
\bes
    B_{t,j}^\mu(r) = (1-\psi_T(r)) \int_0^\infty \eta_j(\lambda) \, m_t^\mu(\lambda) \, \phi_\lambda^{(\alpha_0, \beta_0)}(r) \, |c(\lambda)|^{-2} \, d\lambda.
\ees
Now we apply Lemma \ref{lemma:estimate for lambda^-1 c(lambda)^-1} and Lemma \ref{lemma:estimate for m_t^mu(lambda)}(ii) to write for $j\geq 0$
\begin{align*}
    B_{t,j}^\mu(r)& = (1-\psi_T(r)) \, e^{-\rho (r+t)} \\
    &\quad\quad \int_\R \eta_j(\lambda) \, c^\mu(\lambda) \, a_2^\mu(\lambda,t)
    \left[ \dfrac{e^{i\lambda (t+r)}}{c(-\lambda)} \, a_2(\lambda,r) + \dfrac{e^{i\lambda (t-r)}}{c(\lambda)} \, a_2(-\lambda,r) \right]\, d\lambda \\
    &= \, e^{-\rho (r+t)}\,[I_1(\lambda)+I_2(\lambda)],
\end{align*}
where,
\be\label{eqn: I_1(lambda)}
    I_1(\lambda ) = (1-\psi_T(r))  \int_\R \eta_j(\lambda) \, c^\mu(\lambda) \, a_2^\mu(\lambda,t) \, a_2(-\lambda,r) \,  \dfrac{e^{i\lambda (t-r)}}{c(\lambda)} \, d\lambda,
\ee
and
\be\label{eqn: I_2(lambda)}
    I_2(\lambda ) = (1-\psi_T(r)) \int_\R \eta_j(\lambda) \, c^\mu(\lambda) \, a_2^\mu(\lambda,t) \, a_2(\lambda,r) \, \dfrac{e^{i\lambda (t+r)}}{c(-\lambda)} \, d\lambda.
\ee
Using integration by parts we obtain,
\bes
    I_1(\lambda) = (1-\psi_T(r)) (-1)^{N+1} \int_\R \dfrac{d^{N+1}}{d\lambda^{N+1}} \left( \eta_j(\lambda) \, \frac{c^\mu(\lambda)}{c(\lambda)} \, a_2^\mu(\lambda,t) \, a_2(-\lambda,r) \right) \dfrac{e^{i\lambda (t-r)}}{i^{N+1} (t-r)^{N+1}} \, d\lambda.
\ees
Applying Lemma \ref{lemma:estimate for c^mu(lambda)/c(lambda)} and  Lemma \ref{lemma:estimate for m_t^mu(lambda)} we get for $\lambda \in $ supp($\eta_j$) and for $j\geq 0$,
\begin{align*}
    |I_1(\lambda)| &\leq (1-\psi_T(r)) \int_{\text{supp}(\eta_j)} \frac{2^{-(\operatorname{Re }\mu + N +1)\,j}}{|t-r|^{N+1}}\, d\lambda.
\end{align*}
If $0 < |t-r| \le 1$, then on the support of $1-\psi_T$, we have $|t-r|\geq \frac{1}{2}$. Thus in this case we have \bes |t-r|^{-N-1} < C \, (1+ |t-r|)^{-N}, \ees and if $|t-r| > 1$, we have $|t-r|^{-N-1} < C \, (1+ |t-r|)^{-N}$. Thus,
\bes
    |I_1(\lambda)| \le C \, 2^{-(\operatorname{Re }\mu + N)\,j} \, (1+|t-r|)^{-N}.
\ees
Similarly, we can show that 
\bes
    |I_2(\lambda)| \le C \, 2^{-(\operatorname{Re }\mu + N)\,j} \, (1+|t-r|)^{-N}.
\ees
Hence, we have for $j\geq 0$
\be\label{eqn: estimate for B_{t,j}^mu}
    | B_{t,j}^\mu(r) | \le C \, 2^{-(\operatorname{Re }\mu + N)\,j} \, e^{-\rho(r+t)} \, (1+|t-r|)^{-N}.
\ee
Therefore, by Plancherel's theorem,
\bes
    \|f * B_{t,j}^\mu\|_2 = \|\mathcal{F}(f)\, \mathcal{F}(B_{t,j}^\mu)\|_2 \le \|\mathcal{F}(B_{t,j}^\mu)\|_\infty \, \|f\|_2
\ees
Now using Lemma \ref{lemma:estimate for m_t^mu(lambda)}(i) and the estimate (\ref{eqn: estimate for B_{t,j}^mu}) for $B_{t,j}^\mu$  we have,
\begin{align*}
    |\mathcal{F}(B_{t,j}^\mu)(\lambda)| & \leq C \int_0^\infty |B_{t,j}^\mu(r)| \, |\phi_\lambda^{(\alpha_0,\beta_0)}(r)| \, \Delta_{\alpha_0,\beta_0}(r) \, dr  \\ 
    &\leq C \int_0^\infty |B_{t,j}^\mu(r)|\, (1+r) \,e^{-\rho r}\, (\sinh r)^{2\alpha_0+1} \, (\cosh r)^{2\beta_0 +1} \, dr \\
    &\leq C \, 2^{-(\operatorname{Re }\mu + N)\, j} \, e^{-\rho t } \int_{T+3/2}^\infty (1+r) \, (1+|t-r|)^{-N} \, dr.
\end{align*}
It can be shown that for $N\geq 3$, 
\bes
    \int_{T+3/2}^\infty (1+r) \, (1+|t-r|)^{-N} \, dr \le C \, (1+t).
\ees
Thus, we obtain
\be\label{eqn: f * B_t,j^mu}
     \|f * B_{t,j}^\mu\|_2  \le C \, 2^{-(\operatorname{Re }\mu + N)\, j} \, e^{-\rho T } \, (1+T) \, \|f\|_2.
\ee
Hence from (\ref{eqn: f * (A_t,j^mu + B_t,j^mu)}) and (\ref{eqn: f * B_t,j^mu}), 
\begin{align*}
    \|f* A_{t,j}^\mu\|_2 
    &\le \|f * (A_{t,j}^\mu + B_{t,j}^\mu)\|_2  + \|f * B_{t,j}^\mu\|_2 \\ 
    &\le C \left( 2^{-(\text{Re}\,\mu + \alpha_0 + \frac{1}{2})\,j} + 2^{-(\operatorname{Re }\mu + N)\, j}   \right) 
    e^{-\rho T} \, (1+T) \, \|f\|_2\\
    &\leq C \, 2^{-(\operatorname{Re }\mu + \alpha_0 + \frac{1}{2})\,j} \, e^{-\rho T} \, (1+T) \, \|f\|_2 \quad \text{for } N > \alpha_0+\frac{1}{2}.
\end{align*}
Now we need to show the estimates (\ref{lemma:estimate_2}). The proof is similar to the above; the only difference is that differentiation with respect to $t$ may bring down an extra factor of $\lambda \asymp 2^j$. First, observe that 
\begin{align*}
    \mathcal{F} \left( \frac{\partial}{\partial t} \left( A_{t,j}^\mu + B_{t,j}^\mu \right) \right)(\lambda) 
    &= \int_{X} \frac{\partial}{\partial t} \left( A_{t,j}^\mu + B_{t,j}^\mu \right)(x)\, \phi_\lambda^{(\alpha_0, \beta_0)}(x) \, dx \\ 
    &= \frac{\partial}{\partial t} \int_{X} \left( A_{t,j}^\mu + B_{t,j}^\mu \right)(x)\, \phi_\lambda^{(\alpha_0, \beta_0)}(x) \, dx \\
    &= \partial_t \left( \mathcal{F} \left( A_{t,j}^\mu + B_{t,j}^\mu \right)(\lambda) \right).
\end{align*}
By Plancherel's theorem,
\begin{align*}
    \left\| \partial_t \left( f * \left( A_{t,j}^\mu + B_{t,j}^\mu \right) \right) \right\|_2 
    &= \left\| \mathcal{F}(f) \, \mathcal{F} \left( \partial_t \left( A_{t,j}^\mu + B_{t,j}^\mu \right) \right) \right\|_2 \\
    &= \left\| \mathcal{F}(f) \, \partial_t \mathcal{F} \left( A_{t,j}^\mu + B_{t,j}^\mu \right) \right\|_2 \\
    &= \left\| \mathcal{F}(f) \, \partial_t \left( \eta_j\,m_t^\mu \right) \right\|_2 \\
    &\le  \| \partial_t \left(\eta_j\,m_t^\mu \right) \|_\infty \, \|f\|_2.
\end{align*}
From Lemma \ref{lemma:estimate for m_t^mu(lambda)}(ii), for $j>0$ we have,
\begin{align*}
    \partial_t(m_t^\mu(\lambda)) &= - \rho \, m_t^\mu(\lambda) + e^{-\rho t} \left(  i\, \lambda\, e^{i\lambda t}\, c^\mu(\lambda) \, a_2^\mu(\lambda,t) 
    - i\, \lambda\, e^{-i\lambda t} \, c^\mu(-\lambda) \, a_2^\mu(-\lambda,t)\right) \\
    &\quad\quad\quad + e^{-\rho t} \left( e^{i\lambda t} \, c^\mu(\lambda) \, \partial_t(a_2^\mu(\lambda,t)) 
    + e^{-i\lambda t} \,c^\mu(-\lambda) \, \partial_t(a_2^\mu(-\lambda,t)) \right).
\end{align*}
Using Lemma \ref{lemma:estimate for lambda c^mu(lambda)} and Lemma \ref{lemma:estimate for m_t^mu(lambda)}(ii), we get for $\lambda \in $  supp($\eta_j$) and for $j>0$,  
\begin{align*}
    \left|\partial_t(m_t^\mu(\lambda)) \right| &\leq  \rho \, \left|m_t^\mu(\lambda)\right| + e^{-\rho t} \left( \left| \lambda\,c^\mu(\lambda) \right|\, \left| a_2^\mu(\lambda,t)\right| 
    + \left| \lambda \, c^\mu(-\lambda)\right| \, \left| a_2^\mu(-\lambda,t)\right| \right) \\
    &\quad\quad\quad + e^{-\rho t} \, \frac{1}{|\lambda|} \left( \left|\lambda\, c^\mu(\lambda)\right| \, \left| \partial_t(a_2^\mu(\lambda,t))\right| 
    + \left| \lambda\, c^\mu(-\lambda) \right| \, \left| \partial_t(a_2^\mu(-\lambda,t))\right| \right) \\
    &\leq C \, e^{-\rho t} \left(  2^{-(\text{Re}\,\mu + \alpha_0 + \frac{1}{2})\,j} + 2^{-(\text{Re}\,\mu + \alpha_0 - \frac{1}{2})\,j} + 2^{-(\text{Re}\,\mu + \alpha_0 + \frac{1}{2})\,j} \right) \\
    & \le C \, 2^{-(\text{Re}\,\mu + \alpha_0 - \frac{1}{2})\,j} \, e^{-\rho t}.
\end{align*}
Hence for $j>0$,
\bes
    \left\| \partial_t \left( f * \left( A_{t,j}^\mu + B_{t,j}^\mu \right) \right) \right\|_2 \le C \, 2^{-(\text{Re}\,\mu + \alpha_0 - \frac{1}{2})\,j} \, e^{-\rho T} \, \|f\|_2.
\ees
For $j=0$, we will use Lemma \ref{lemma:estimate for m_t^mu(lambda)}(i) to obtain 
\bes
      \left\| \partial_t \left( f * \left( A_{t,0}^\mu + B_{t,0}^\mu \right) \right) \right\|_2 \leq C \, e^{-\rho T} \,(1+T) \|f\|_2.
\ees
Thus, for any $t\in[T,T+1]$ and $j\ge 0$, 
\bes
    \left\| \partial_t \left( f * \left( A_{t,j}^\mu + B_{t,j}^\mu \right) \right) \right\|_2 \leq C \, 2^{-(\text{Re}\,\mu + \alpha_0 - \frac{1}{2})\,j} \, e^{-\rho T} \, (1+T) \, \|f\|_2.
\ees
On the other hand, we have for $j\geq 0$,
\begin{align*}
    B_{t,j}^\mu(r)& = (1-\psi_T(r)) \, e^{-\rho (r+t)} \\
    &\quad\quad \int_\R \eta_j(\lambda) \, c^\mu(\lambda) \, a_2^\mu(\lambda,t)
    \left[ \dfrac{e^{i\lambda (t+r)}}{c(-\lambda)} \, a_2(\lambda,r) + \dfrac{e^{i\lambda (t-r)}}{c(\lambda)} \, a_2(-\lambda,r) \right]\, d\lambda.
\end{align*}
Therefore,
\bes 
    \partial_t(B_{t,j}^\mu(r)) = -\rho\,B_{t,j}^\mu(r) +  e^{-\rho (r+t)}\,[J_1(\lambda)+ J_2(\lambda)],
\ees
where,
\begin{align*}
    J_1(\lambda ) &= (1-\psi_T(r)) \\
    &\quad\quad\int_\R \eta_j(\lambda) \, c^\mu(\lambda) \, \partial_t(a_2^\mu(\lambda,t)) \left[ \dfrac{e^{i\lambda (t+r)}}{c(-\lambda)}\, a_2(\lambda,r) + \dfrac{e^{i\lambda (t-r)}}{c(\lambda)}\, a_2(-\lambda,r) \right]\, d\lambda,
\end{align*}
and
\begin{align*}
    J_2(\lambda ) &= (1-\psi_T(r)) \\
    &\quad\quad\int_\R \eta_j(\lambda) \, c^\mu(\lambda) \, a_2^\mu(\lambda,t) \, i\lambda \left[ \dfrac{e^{i\lambda (t+r)}}{c(-\lambda)}\, a_2(\lambda,r) + \dfrac{e^{i\lambda (t-r)}}{c(\lambda)}\, a_2(-\lambda,r) \right]\, d\lambda.
\end{align*}
We have already shown the estimate (\ref{eqn: estimate for B_{t,j}^mu}) of $B_{t,j}^\mu$. Now using integration by parts and applying Lemma \ref{lemma:estimate for c^mu(lambda)/c(lambda)} and Lemma \ref{lemma:estimate for m_t^mu(lambda)} we get, 
\bes
    |J_1(\lambda)| \le C \, 2^{-(\operatorname{Re }\mu + N)\,j} \, (1+|t-r|)^{-N},
\ees
and
\bes
    |J_2(\lambda)| \le C \, 2^{-(\operatorname{Re }\mu + N-1)\,j} \, (1+|t-r|)^{-N}.
\ees
Finally, we obtain
\be\label{eqn: estimate for partial_t(B_{t,j}^mu)}
    | \partial_t(B_{t,j}^\mu(r)) | \le C \, 2^{-(\operatorname{Re }\mu + N-1)\,j} \, e^{-\rho(r+t)} \, (1+|t-r|)^{-N}.
\ee
Therefore, by Plancherel's theorem,
\bes
    \|\partial_t(f * B_{t,j}^\mu)\|_2 = \|\mathcal{F}(f)\, \mathcal{F}(\partial_t(B_{t,j}^\mu))\|_2 \le \|\mathcal{F}(\partial_t(B_{t,j}^\mu))\|_\infty \, \|f\|_2.
\ees
Now using Lemma \ref{lemma:estimate for m_t^mu(lambda)}(i) and the estimate (\ref{eqn: estimate for partial_t(B_{t,j}^mu)}) for $\partial_t(B_{t,j}^\mu)$  we have for $j\geq 0$,
\begin{align*}
    |\mathcal{F}(\partial_t(B_{t,j}^\mu))(\lambda)| & \leq C \int_0^\infty |\partial_t(B_{t,j}^\mu(r))| \, |\phi_\lambda^{(\alpha_0,\beta_0)}(r)| \, \Delta_{\alpha_0,\beta_0}(r) \, dr  \\ 
    &\leq C \int_0^\infty |\partial_t(B_{t,j}^\mu(r))|\, (1+r) \,e^{-\rho r}\, (\sinh r)^{2\alpha_0+1} \, (\cosh r)^{2\beta_0 +1} \, dr \\
    &\leq C \, 2^{-(\operatorname{Re }\mu + N -1)\, j} \, e^{-\rho t } \int_{T+3/2}^\infty (1+r) \, (1+|t-r|)^{-N} \, dr\\
    &\leq C \, 2^{-(\operatorname{Re }\mu + N-1)\, j} \, e^{-\rho t } (1+t) \quad \text{for } N\ge 3.
\end{align*}
Thus, we obtain
\bes
     \|\partial_t(f * B_{t,j}^\mu)\|_2  \le C \, 2^{-(\operatorname{Re }\mu + N-1)\, j} \, e^{-\rho T } \, (1+T) \, \|f\|_2.
\ees
Hence for $j\geq 0$,
\begin{align*}
    \left\| \partial_t (f * A_{t,j}^\mu) \right\|_2 
    &\le \left\| \partial_t \left( f * (A_{t,j}^\mu + B_{t,j}^\mu) \right) \right\|_2 + \left\| \partial_t (f * B_{t,j}^\mu) \right\|_2 \\
    &\le C \, \left( 2^{-(\text{Re}\,\mu + \alpha_0 - \frac{1}{2})\,j} +  2^{-(\text{Re}\,\mu + N - 1)\,j} \right) e^{-\rho T}\, (1+T)\, \|f\|_2\\
    &\le C\, 2^{-(\text{Re}\,\mu + \alpha_0 - \frac{1}{2})\,j}\,  e^{-\rho T} \,(1+T) \, \|f\|_2
    \quad \text{for }  N>\alpha_0 + \frac{1}{2}.
\end{align*}

By [\cite{sogge2017}, Lemma 2.4.2], for any continuously differentiable function F and for any $t\ge T$, we have
\bes
    |F(t)|^2 \le |F(T)|^2 + C\left(\int_T^t|F(s)|^2\, ds\right)^{1/2}
    \left(\int_T^t |F'(s)|^2\, ds\right)^{1/2}.
\ees
Now fix any $x\in G/K$ and let $F(t) = (f*A_{t,j}^\mu)(x)$, for $t\in [T,T+1]$. Then we have,
\bes
    \sup_{t \in [T, T+1]} |f * A_{t,j}^\mu (x)|^2 \le |f * A_{T,j}^\mu (x)|^2 +\, C \left( \int_T^{T+1} |f * A_{t,j}^\mu (x)|^2 \, dt \right)^{1/2} \left( \int_T^{T+1} |\partial_t (f * A_{t,j}^\mu(x))|^2 \, dt \right)^{1/2}.
\ees
Therefore,
\begin{align*}
    \left\| \sup_{t \in [T, T+1]} |f * A_{t,j}^\mu| \right\|_2^2 &\le \int_{G/K} \, \sup_{t \in [T, T+1]} |f * A_{t,j}^\mu (x)|^2 \, dx \\
    &\le \|f * A_{T,j}^\mu\|_{2}^2 \\
    &\qquad+\,C\,\int_{G/K} \left( \int_T^{T+1} |f * A_{t,j}^\mu (x)|^2 \, dt \right)^{1/2} \left( \int_T^{T+1} |\partial_t (f * A_{t,j}^\mu(x))|^2 \, dt \right)^{1/2} dx\\
    &\le \|f * A_{T,j}^\mu\|_{2}^2 \,+\,C \left( \int_{G/K} \int_T^{T+1} |f * A_{t,j}^\mu (x)|^2 \, dt \, dx \right)^{1/2} \\
    &\hspace{1.5in}\left( \int_{G/K} \int_T^{T+1} |\partial_t (f * A_{t,j}^\mu(x))|^2 \, dt \, dx \right)^{1/2}\\
    &= \|f * A_{T,j}^\mu\|_{2}^2 \,+\,C \left( \int_T^{T+1} \|f * A_{t,j}^\mu\|_2^2 \, dt \right)^{1/2} \left( \int_T^{T+1} \|\partial_t (f * A_{t,j}^\mu)\|_2^2 \, dt \right)^{1/2}.    
\end{align*}
Hence,
\bes
    \left\| \, \sup_{t \in [T, T+1]} |f * A_{t,j}^\mu| \, \right\|_2 \le C \, 2^{-(\operatorname{Re } \mu + \alpha_0)\,j} \, e^{-\rho T} (1+T) \, \|f\|_2.
\ees
\end{proof}

We now prove the following local estimate.
\begin{lemma}
\label{local-reduction}
Let $X=G/K$ be a rank-one Riemannian symmetric space of noncompact
type and let $d=\dim X$. There exists $t_0>0$ such that, for
$0<t<t_0$, 
\[
 \left\|
     \sup_{0<t<t_0}|\mathcal M_t^\mu f|
 \right\|_{L^2(X)}
 \leq C_\mu \, \|f\|_{L^2(X)},
\]
for
\[
        \operatorname{Re}\,\mu >1-\frac d2.
\]

\end{lemma}

\begin{proof}
We fix $x_0\in X$ and choose a
normal coordinate neighborhood $U$ of $x_0$. By taking $t_0>0$
sufficiently small, we may assume that $B(x,t)$ remains in a fixed
coordinate neighborhood whenever $x\in U$ and $0<t<t_0$.

For $\operatorname{Re}\mu>0$, the kernel defining $\mathcal M_t^\mu$ contains the
factor
\[
 \frac{1}{\Gamma(\mu)}
 \bigl(\cosh(2t)-\cosh(2d(x,y))\bigr)_+^{\mu-1}.
\]
Writing $r=d(x,y)$ and using
\[
 \cosh(2t)-\cosh(2r)
   =2\sinh(t+r)\sinh(t-r),
\]
we obtain
\[
 \cosh(2t)-\cosh(2r)
   =2(t+r)(t-r)b(t,r),
\]
where
\[
 b(t,r)
 =
 \frac{\sinh(t+r)}{t+r}
 \frac{\sinh(t-r)}{t-r}
\]
is smooth and strictly positive for $t,r$ sufficiently small.
Hence, modulo a smooth nonvanishing factor, the singular part of the
kernel is
\[
       \frac{(t-d(x,y))_+^{\mu-1}}{\Gamma(\mu)}.
\]

We next use geodesic normal coordinates centered at $x$ and write
\[
        y=\exp_x(tz).
\]
Then
\[
        d(x,y)=t|z|,
        \qquad
        dy=t^d J(x,t,z)\,dz,
\]
where $J$ is smooth and nonvanishing for small $t$. After incorporating
the normalization occurring in the definition of $\mathcal M_t^\mu$,
as well as the Jacobian and the preceding smooth factors, a localized
piece of $\mathcal M_t^\mu$ takes the form
\be\label{lem:local-reduction}
 T_t^\mu f(x)
 =
 \int
 f\bigl(\exp_x(tz)\bigr)
 A_\mu(t,x,z)
 \frac{(1-|z|)_+^{\mu-1}}{\Gamma(\mu)}
 \,dz,
 \ee
where $A_\mu$ is smooth in $(t,x,z)$ and depends analytically on
$\mu$.

The singular hypersurface in \eqref{lem:local-reduction} is
\[
                    |z|=1.
\]
Cover $S^{d-1}$ by finitely many coordinate patches. On each such
patch, after a rotation of coordinates if necessary, the hypersurface
can be represented as
\[
             z_d=h(t,x;z'),
             \qquad z'=(z_1,\ldots,z_{d-1}).
\]
At $t=0$ this is a piece of the Euclidean unit sphere. For example,
on the upper hemisphere,
\[
             h(0,x;z')=(1-|z'|^2)^{1/2}.
\]
Its Hessian is nondegenerate:
\[
 \det\left(
 \frac{\partial^2h(0,x;z')}
      {\partial z'_j\partial z'_k}
 \right)\neq0.
\]
By continuity, after decreasing $t_0$ if necessary, there is a
constant $c>0$ such that
\[
 \left|
 \det\left(
 \frac{\partial^2h(t,x;z')}
      {\partial z'_j\partial z'_k}
 \right)
 \right|\ge c
\]
uniformly on the supports of the localized amplitudes. Thus the
corresponding hypersurfaces have uniformly nonvanishing Gaussian
curvature.

Furthermore, since the defining function $1-|z|$ vanishes simply on
$|z|=1$, on each coordinate patch we may write
\[
  1-|z|
   =
  q(t,x,z)\bigl(h(t,x;z')-z_d\bigr),
\]
where $q$ is smooth and nonvanishing. Choosing the orientation so
that $q>0$, we have
\[
 \frac{(1-|z|)_+^{\mu-1}}{\Gamma(\mu)}
 =
 q(t,x,z)^{\mu-1}
 \frac{(h(t,x;z')-z_d)_+^{\mu-1}}{\Gamma(\mu)}.
\]
The factor $q^{\mu-1}$ is smooth and may therefore be absorbed into
the amplitude. Consequently, every localized piece is of the form
\bes
 T_t^\mu f(x)
 =
 \int
 f\bigl(\Theta(t,x,z)\bigr)
 a_\mu(t,x,z)
 \frac{(h(t,x;z')-z_d)_+^{\mu-1}}{\Gamma(\mu)}
 \,dz,
 \ees
with
\[
 \det h_{z'z'}(t,x;z')\neq0.
\]
This is the local analytic family of generalized hypersurface means
appearing in Sogge's maximal theorem \cite[Theorem~2.2]{Sogge1987}.

By \cite[Theorem~2.2]{Sogge1987}, the corresponding local maximal
operator is bounded on $L^2$ whenever
\[
              \operatorname{Re}\mu>1-\frac d2.
\]
Thus, for every sufficiently small coordinate ball $U$,
\be\label{lem:local-reduction-3}
 \left\|
    \sup_{0<t<t_0}|T_t^\mu f|
 \right\|_{L^2(U)}
 \leq C_\mu
 \|f\|_{L^2(U^*)},
\ee
where $U^*$ is a fixed enlargement of $U$.

Finally, since $X=G/K$ is homogeneous, the coordinate neighborhoods
and all the constants above can be chosen uniformly over $X$.
Choose a uniformly locally finite covering $\{U_j\}$ of $X$ by such
balls, together with enlarged balls $\{U_j^*\}$ having uniformly
bounded overlap. Summing the squares of
\eqref{lem:local-reduction-3} gives
\[
 \begin{aligned}
 \left\|
   \sup_{0<t<t_0}|\mathcal M_t^\mu f|
 \right\|_{L^2(X)}^2
 &\leq C
 \sum_j \|f\|_{L^2(U_j^*)}^2  \\
 &\leq C
 \|f\|_{L^2(X)}^2.
 \end{aligned}
\]
Therefore
\[
 \left\|
   \sup_{0<t<t_0}|\mathcal M_t^\mu f|
 \right\|_{L^2(X)}
 \lesssim_\mu
 \|f\|_{L^2(X)}
\]
provided
\[
       \operatorname{Re}\mu>1-\frac d2.
\]
Since
\[
       \alpha_0=\frac{d-2}{2},
\]
the result follows.
\end{proof}

\begin{lemma}\label{l2-est}
For $\operatorname{Re }\mu > -\alpha_0$ we have
    \bes
        \left\| \, \sup_{t > 0} |\mathcal M_{t}^\mu(f)| \, \right\|_2 \le C \, \|f\|_2.
    \ees
\end{lemma}
\begin{proof}
Let $t_0>0$ be fixed, as in Lemma \ref{local-reduction}. Then, by Lemma \ref{lemma: sup_t[T,T+1] |f*A_t,j^mu|},
\begin{align*}
    \left\| \sup_{t \ge t_0} |\mathcal M_{t}^\mu(f)| \right\|_2 
    &\le \sum_{T=t_0}^\infty \left\| \sup_{t \in [T, T+1]} |\mathcal M_{t}^\mu (f)| \right\|_2 \\
    &\le \sum_{T=t_0}^\infty \sum_{j = 0}^\infty \left\| \sup_{t \in [T, T+1]} |f * A_{t,j}^\mu| \right\|_2 \\
    &\le C \,\|f\|_2 \sum_{T=t_0}^\infty e^{-\rho T} \,(1+T) \left( \sum_{j = 0}^\infty 2^{-(\text{Re}\,\mu + \alpha_0)\,j} \right) \\
    &\le C \, \|f\|_2 \sum_{T=t_0}^\infty e^{-\rho T}\, (1+T) \quad \text{as } \operatorname{Re }\mu > -\alpha_0\\
    &\le C \, \|f\|_2 \quad \text{as } \rho>0.
\end{align*}
 Therefore, the lemma follows from Lemma  \ref{local-reduction}.  
\end{proof}

Now we are in a position to prove our main theorem, Theorem \ref{main-thm}.    
\begin{proof}[Proof of Theorem \ref{main-thm}]
Let $f^\ast$  be the centered Hardy-Littlewood maximal function on $G/K$ which is defined by \bes f^*(x) = \sup_{t > 0} \frac{1}{|B(x, t)|} \left| \int_{B(x, t)} f(y) \, \mathrm{d}\sigma(y) \right|.\ees 
First, we will show that for $\operatorname{Re}\,\mu \geqslant 1, \, \mathcal M^\mu f(x) \le C f^*(x)$. We observe that
\bes
    |\mathcal M_{t}^\mu f(x)|\leq \frac{1}{|\Gamma(\mu)|} \, |C(t, \mu)| \int_{B(x, t)} |f(y)| \left|\cosh 2t - \cosh 2d(x,y)\right|^{\operatorname{Re}\,\mu - 1} d\sigma(y).
    \ees 
Since, $\cosh 2t - \cosh 2d(x,y) \le \cosh 2t - 1 = 2 \sinh^2 t$. So, for $\operatorname{Re } \mu \ge 1$,
\bes
    \left(\cosh 2t - \cosh 2d(x,y)\right)^{\operatorname{Re }\mu - 1} \le 2^{\operatorname{Re }\mu -1} \sinh^{2\operatorname{Re }\mu-2}(t).
\ees
Using the expression,
\bes 
C(t,\mu) = \frac{4 \, e^{2\mu t}}{\sinh^{2\alpha_0 + 2 \mu}(t) \, \cosh^{2 \beta_0 + 2\mu}(t)}
\ees we get
\begin{align*}
    |\mathcal M_{t}^\mu f(x)| &\leq \frac{1}{|\Gamma(\mu)|} \,  |C(t,\mu)| \, 2^{\operatorname{Re }\mu -1} \sinh^{2\operatorname{Re }\mu-2}(t) \int_{B(x, t)} |f(y)| \, d\sigma(y)\\ \\
     &\le C \frac{1}{|B(x, t)|} \int_{B(x, t)} |f(y)| \, d\sigma(y)\\ \\
   &\le C  f^*(x).
\end{align*}
So, for $\operatorname{Re}\,\mu \geqslant 1$, 
\bes
\mathcal M^\mu f(x) \le C f^*(x).\ees

We now use the fact that $f \to f^* \text{ is a bounded operator on } L^p(G/K) \text{ for all } 1 < p \leqslant \infty$ (\cite{Stromberg1981}).
Using this, it is clear that $\mathcal M^\mu$ is a bounded operator on $L^p(G/K)$ for all $ \ 1 < p \le \infty$. 
In particular, $\mathcal M^\mu$ is a bounded operator on $L^{1+r}(G/K)$  for $r>0$ and $\operatorname{Re}\mu\geq 1$.
Also, from Lemma \ref{l2-est}, $\mathcal M^{\mu}$ is a bounded operator on $L^2(G/K)$ whenever $\operatorname{Re}(\mu) > -\alpha_0$. 

Now, using complex interpolation, we obtain that
\begin{center}
    $\mathcal M^{\mu}$ is a bounded operator on $L^p(G/K)$ for $1 < p \leq 2$  and $\operatorname{Re}\,\mu > -\alpha_0 + (\alpha_0 + 1)\left(\frac{2}{p} - 1\right).$ 
\end{center}
We observe that \bes -\alpha_0 + (\alpha_0 + 1)\left(\frac{2}{p} - 1\right)=1-d+ \frac{d}{p}.\ees This concludes the proof of $(1)$.

Now we will show for $\operatorname{Re }\mu >0$, $\mathcal M^{\mu} \text{ is a bounded operator on  } L^\infty(G/K)$. Clearly,
$$
    |\mathcal M_{t}^\mu f(x)| \le \frac{|C(t,\mu)|}{|\Gamma(\mu)|} \, \|f\|_\infty \, I(t)
$$
where 
$$
    I(t)  = \int_0^t  \left( \cosh 2t - \cosh 2s \right)^{\operatorname{Re}\mu - 1} \sinh^{2\alpha_0 +1}(s) \cosh^{2\beta_0+1}(s) \, ds
$$
For small $t$, 
$$I(t) \leq  C \int_0^t  \left(t^2 - s^2 \right)^{\operatorname{Re}\mu - 1} s^{2\alpha_0 +1} \, ds =  C \, t^{2\operatorname{Re}\mu + 2\alpha_0} \,B\left(\operatorname{Re}\mu, \alpha_0+1\right),$$
which is finite for $\operatorname{Re } \mu > 0$. Also, for small $t$, \bes |C(t,\mu)| \leq \frac{C}{t^{2\operatorname{Re}\mu + 2\alpha_0}}.
\ees
Now for large $t$,
$$
    I(t)  \leq C \int_0^t  \left( e^{2t} - e^{2s} \right)^{\operatorname{Re}\mu - 1} e^{2\rho s} \, ds \le  C\, (e^t)^{2\operatorname{Re}\mu+2\rho-2} \,B(\operatorname{Re}\mu,\rho),
$$
which is finite for $\operatorname{Re } \mu > 0$. Also, for large $t$, \bes |C(t,\mu)| \leq \frac{C}{(e^t)^{2\operatorname{Re}\mu + 2\rho - 2}}.\ees
So, finally we get 
$$
    |\mathcal M_{t}^\mu f(x)| \le C\, \|f\|_\infty.
$$
So, for $\operatorname{Re }\mu >0$, $\mathcal M^{\mu} \text{ is a bounded operator on  } L^\infty(G/K)$. Now, interpolate between $p=2$, $\operatorname{Re } \mu > - \alpha_0$ and $p=\infty$, $\operatorname{Re } \mu >0$, we get
$$
\mathcal M^{\mu} \text{ is a bounded operator on  } L^p(G/K) \text{ for }
2 \le p \le \infty  \text{ and } \operatorname{Re}\,\mu > -\frac{2\alpha_0}{p}.
$$
Again we observe that 
\bes
-\frac{2\alpha_0}{p}=\frac{2-d}{p}.
\ees
This completes the proof.
\end{proof}

As a corollary, we recover the boundedness of the usual spherical maximal function. We recall  
$$
\mathcal{M}f(x) = \mathcal{M}^0 f(x) = \sup_{t > 0} \left|\mathcal M_{t}^{0} f(x) \right|  = \sup_{t > 0} \left| \int_K f(u k a_{t} 1) \, dk  \right|.
$$
Then 
\begin{corollary}
    For $d\ge 3$, $\mathcal{M}$ is a bounded operator on $L^p(G/K)$ for $p>\frac{d}{d-1}.$
\end{corollary}

Now, we will prove the necessary condition. The proof is mostly in the line of proof of Chen--Shen--Wang--Yan \cite{ChenShenWangYan2025}.
\begin{proof}[Proof of Theorem \ref{thm:necessary}]
    Let $\Gamma$ be the cone vertexed at $eK$ and tangent to the ball $B_{3c_1}(a_{1/2})$, where $c_1$ is a constant to be determined.
    For sufficiently small $\delta>0$, define $f_\delta$ by
    $$
        f_\delta(y)= \begin{cases}
        \dfrac{|y|^{1-d-\mu}}{-\log |y|} & \text{in } (B_{1/2}(eK) \setminus B_\delta(eK)) \cap \Gamma, \\[2ex]
        0 & \text{outside } (B_{1/2}(eK) \setminus B_\delta(eK)) \cap \Gamma.
        \end{cases}
    $$
    We want to estimate $ \mathcal M^\mu_{|x|}(f_\delta)(x)$ for $x \in B_{c_1}(a_{1/2})$. To do this, we let $\widetilde{S}_x \subset \Gamma$ be a cone vertexed at $eK$ and tangent to $B_{c_1}(x)$. Set $S_x = B_{1/10}(eK) \cap \widetilde{S}_x$. We choose $c_1$ sufficiently small such that for all sufficiently small $\delta > 0$ and $x \in B_{c_1}(a_{1/2})$, $S_x \cap \text{supp} f_\delta$ is contained in the interior of $B_{|x|}(x)$, and that for $y \in S_x \cap \text{supp} f_\delta$ we have
    \bes
        |x| - d(x,y) \sim |y| >0.
    \ees
    Therefore, for such $x$ we have
    \beas
        | \mathcal M_{|x|}^\mu(f_\delta)(x)|
        &\ge& C_\mu \int_{S_x \cap \text{ supp } f_\delta} (\cosh 2|x| - \cosh 2d(x, y))^{\mu-1} \, \frac{|y|^{1-d-\mu}}{-\log|y|} \, dy \\
        &\ge& C_\mu \int_\delta^{1/10} r^{\mu-1} \frac{r^{1-d-\mu}}{-\log r} r^{d-1} \, dr \\
        &\ge& C_\mu \log \log(1/\delta)
    \eeas
    using polar coordinates. Thus,
    \bes
        \|\mathcal M^\mu(f_\delta)\|_p \ge \left( \int_{B_{c_1}(a_{1/2})} |\mathcal M_{|x|}^\mu(f_\delta)(x)|^p \, dx \right)^{1/p} \ge C_\mu \log \log(1/\delta).
    \ees
    Now,
    \beas
        \|f_\delta\|_p^p &=& \int_{(B_{1/2}(eK) \setminus B_\delta(eK)) \cap \Gamma} \dfrac{|y|^{(1-d-\mu)p}}{(-\log |y|)^p} \, dy\\
        &\leq & \int_{\delta}^{1/2} \dfrac{r^{(1-d-\mu)p}}{(-\log r)^p} \, r^{d-1} \, dr \\
       &=& \int_{\delta}^{1/2} \dfrac{r^{\beta}}{(-\log r)^p} \, dr
    \eeas
    using polar coordinates, and where $\beta = (1-d-\mu)p + d -1$. If $\mu \le 1 - d + \frac{d}{p}$, then $\beta \ge -1$. So if $\beta > -1$,
    \bes
        \int_{\delta}^{1/2} \dfrac{r^{\beta}}{(-\log r)^p} \, dr
        \le \frac{1}{(\log 2)^{p}} \int_{0}^{1/2} r^\beta \, dr < \infty
    \ees
    and if $\beta = -1$, using the substitution $-\log r = z$, we have
    \bes
        \int_{\delta}^{1/2} \dfrac{1}{r(-\log r)^p} \, dr
        =  \int_{\log 2}^{-\log \delta}  \dfrac{dz}{z^p}
        \leq \int_{\log 2}^{\infty}  \dfrac{dz}{z^p} < \infty \quad (\text{ as } p > 1).
    \ees
    Now if $\mu > 1 - d + \frac{d}{p}$, then $\beta < -1$. So,
    \bes
        \int_{\delta}^{1/2} \dfrac{r^{\beta}}{(-\log r)^p} \, dr
        \le \frac{1}{(\log 2)^{p}} \int_{\delta}^{1/2} r^{\beta} \, dr
        \le C \, \delta^{\beta +1}.
    \ees
    Therefore,
    \bes
        \|f_\delta\|_p \le
        \begin{cases}
            C_\mu & \text{ for } \mu \leq 1 - d + \frac{d}{p}, \\[2ex]
            C_\mu \, \delta ^{1-d-\mu + \frac{d}{p}} & \text{ for } \mu > 1 - d + \frac{d}{p}.
        \end{cases}
    \ees
Therefore, \bes \|\mathcal M^\mu(f_\delta) \|_p \leq C\, \|f_\delta\|_p, \ees for all $\delta>0$, implies  \bes \mu > 1 - d + \frac{d}{p}.\ees
\end{proof}

\section{lacunary spherical maximal function}
This section is dedicated to the proof of Theorem \ref{thm: L_*mu is bdd op}. To prove the theorem we mainly adapt the method of \cite{WangZhang}.
We consider the following two operators:
\bes
    \fg_*^\mu(f)(x) := \sup_{j\in \Z_+} \left|  \mathcal M_j^\mu (f)(x)\right| \quad  \text{and} \quad 
    \fl_*^\mu(f)(x) := \sup_{j\in \Z_+} \left|  \mathcal M_{2^{-j}}^\mu (f)(x)\right|.
\ees
Since
\bes
    \cL_*^\mu(f)(x) \le \fg_*^\mu(f)(x) + \fl_*^\mu(f)(x),
\ees
for all $f$ and for a.e. $x \in G/K$, we will  first establish the $L^p$-boundedness of the operators $\fg_*^\mu$ and $\fl_*^\mu$ separately, then derive Theorem \ref{thm: L_*mu is bdd op} using an interpolation argument. 
The following proposition establishes the $L^p$-boundedness of the operator $\fg_*^\mu$ for $1< p \le 2$. We recall that $\alpha_0=\frac{m_1 + m_2-1}{2}$.
\begin{proposition} \label{prop: g_*^mu is in L^p for 1<p<=2}
    Suppose that $\mu \in \C$ and $1 < p < 2$. Then the estimate
    \bes
        \| \fg_*^\mu (f) \|_2 \le C \, \| f \|_2,
    \ees
    holds, provided $\operatorname{Re } \mu > -\alpha_0 - \frac{1}{2}$ and
    \bes
        \| \fg_*^\mu (f) \|_p \le C \, \| f \|_p ,
    \ees
    holds, provided $\operatorname{Re } \mu > 0$.
\end{proposition}
\begin{proof}
For $p=2$, using Plancherel's theorem, we have 
\bes
    \|\fg_*^\mu (f)\|_2^2 \le \sum_{j=1}^\infty \|\mathcal M_j^\mu (f)\|_2^2 \le \sum_{j=1}^\infty \|m_j^\mu \|_\infty^2 \|f\|_2^2
\ees
Now using Lemma \ref{lemma:estimate for m_t^mu(lambda)} (i), we have for $\operatorname{Re }\mu > -\alpha_0-1/2$,
\bes
  \|\fg_*^\mu (f)\|_2^2 \le C \left(  \sum_{j=1}^\infty j^2 \, e^{-2\rho j} \right) \|f\|_2^2 \le C \, \|f\|_2^2.  
\ees
Let $1<p<2$. We recall 
\bes
    \mathcal M_{t}^{\mu} f(x) = \int_{G/K} \mathcal{K}'(x,y) f(y) \, dy,
\ees
where,
\bes
    \cK'(x, y) = \frac{1}{\Gamma(\mu)} C(t, \mu) \, \mathbbm{1}_{B(x, t)}(y) \left( \cosh 2t - \cosh 2 d(x, y) \right)^{\mu-1}
\ees
is the radial kernel provided that $\operatorname{Re }\mu >0$. We also observe that $\cK'$ depends on $t, \mu$ and the geodesic distance $d$. We define $\cK_t^\mu(d(x,y))$ for the radial kernal $\cK'(x, y)$. In polar coordinate, we have
\be \label{eqn: kernel k_t^mu(r)}
   \cK_t^\mu(r) =  \frac{1}{\Gamma(\mu)} C(t, \mu) \, \left( \cosh 2t - \cosh 2 r \right)^{\mu-1} \, \mathbbm{1}_{[0,t]}(r).
\ee
We also have the identity
\be \label{eqn: cosh 2t - cosh 2r}
    \cosh 2t - \cosh 2r = \frac{e^{2t}}{2} (1-e^{2r-2t}) (1-e^{-2r-2t})
\ee
and for all $t\ge 1$, from the expression (\ref{eqn: c(t,mu)}) of $C(t, \mu)$  we have,
\bes
    |C(t,\mu)| \le \frac{C}{e^{(2\alpha_0 + 2\beta_0 + 2 \operatorname{Re }\mu)\,t}}.
\ees
Therefore, for all $t\ge 1$,
\begin{align*}
    |\cK_t^\mu(r)| &\le C \, e^{-2 \rho t} (1- e^{2r - 2t})^{\operatorname{Re }\mu -1} \, \mathbbm{1}_{[0,t]}(r) \\
    &\le C \left(  e^{-2 \rho t} (1- e^{2r - 2t})^{\operatorname{Re }\mu -1} \, \mathbbm{1}_{[0,t-1/2]}(r) + e^{-2 \rho t} (1- e^{2r - 2t})^{\operatorname{Re }\mu -1} \, \mathbbm{1}_{[t-1/2,t]}(r) \right) \\
\end{align*}
Now for $0\le r \le t-1/2$, 
\bes
    e^{-2 \rho t} (1- e^{2r - 2t})^{\operatorname{Re }\mu -1} \, \mathbbm{1}_{[0,t-1/2]}(r) \le C\, e^{-2 \rho (t-1/2)} \, \mathbbm{1}_{[0,t-1/2]}(r)
\ees
and for $t-1/2 \le r \le t$,
\bes
    e^{-2 \rho t} (1- e^{2r - 2t})^{\operatorname{Re }\mu -1} \, \mathbbm{1}_{[t-1/2,t]}(r) \le C \, e^{-2 \rho t} (t - r)^{\operatorname{Re }\mu -1} \, \mathbbm{1}_{[t-1/2,t]}(r)
\ees
So for all $t\ge 1$,
\bes
    |\cK_t^\mu(r)| \le C \left( \cK_t^{\mu,1}(r) +  \cK_t^{\mu,2}(r)\right)
\ees
where 
\bes 
    \cK_t^{\mu,1}(r) = e^{-2 \rho (t-1/2)} \, \mathbbm{1}_{[0,t-1/2]}(r)
    \quad \text{and} \quad
    \cK_t^{\mu,2}(r) = e^{-2 \rho t} (t - r)^{\operatorname{Re }\mu -1} \, \mathbbm{1}_{[t-1/2,t]}(r).
\ees
Now
\begin{align*}
\fg_*^\mu (f)(x) &= \sup_{j \in \Z_+} \left| \mathcal M_j^\mu (f)(x) \right| \\
&= \sup_{j \in \Z_+} \left| \int_{G/K} \cK_j^\mu (d(x, y)) f(y) \, dy \right| \\
&\le \sup_{j \in \Z_+} \int_{G/K} |f(y)| \, \left| \cK_j^\mu (d(x, y)) \right| dy \\
&\le \sup_{j \in \Z_+} \int_{G/K} |f(y)| \, \cK_j^{\mu, 1} (d(x, y)) \, dy + \sup_{j \in \Z_+} \int_{G/K} |f(y)| \, \cK_j^{\mu, 2} (d(x, y)) \, dy \\
&\le \sup_{j \in \Z_+} \, \frac{C}{|B(x, j-1/2)|}\int_{B(x,j-1/2)} |f(y)|  \, dy + \sup_{j \in \Z_+} \int_{G/K} |f(y)| \, \cK_j^{\mu, 2} (d(x, y)) \, dy \\
&\le C \, f^*(x) + \sup_{j \in \Z_+} \int_{G/K} |f(y)| \, \cK_j^{\mu, 2} (d(x, y)) \, dy, 
\end{align*}
where $f^*$ is the centered Hardy-Littlewood maximal function.
Therefore,
\bes
    \|\fg_*^\mu (f)\|_p \le C \|f^*\|_p + \left\|\sup_{j \in \Z_+} \int_{G/K} |f(y)| \, \cK_j^{\mu, 2} (d(x, y)) \, dy \right\|_p.
\ees
We now use the fact that $f \to f^*$ is a bounded operator on $L^p(G/K)$ for all $1 < p \le \infty$ (\cite{Stromberg1981}). On the other hand using Lemma \ref{lemma: Kunze-Stein phenomenon},
\begin{align*}
    \left\|\sup_{j \in \Z_+} \int_{G/K} |f(y)| \, \cK_j^{\mu, 2} (d(x, y)) \, dy \right\|_p^p
    &\le \sum_{j =1}^{\infty} \left\|\int_{G/K} |f(y)| \, \cK_j^{\mu, 2} (d(x, y)) \, dy \right\|_p^p \\
    &\le C \sum_{j=1}^\infty e^{-2\rho (p-1)j} \left( \int_{j-1/2}^{j} (j-r)^{\operatorname{Re }\mu -1} \, dr \right)^p \|f\|_p^p \\
    &\le C\, \|f\|_p^p,
\end{align*}
provided that $\operatorname{Re }\mu>0$ and $1<p<2$. This completes the proof of the proposition.
\end{proof}
Next, we will show that the operator $\fl_*^\mu$ defined by
\bes 
    \fl_*^\mu (f)(x) = \sup_{j \in \Z_+} \left|\mathcal M_{2^{-j}}^\mu(f)(x)\right|,
\ees
is bounded from $L^2(G/K)$ to $L^2(G/K)$ and from $L^1(G/K)$ to $L^{1,\infty}(G/K)$, respectively.

\begin{proposition} \label{prop: l_*^mu is bdd on L^2}
    Let $\mu \in \C $ with $\operatorname{Re}\mu > -\alpha_0-1/2$. Then
    \bes
        \|\fl_*^\mu (f)\|_2 \le C \, \|f\|_2.
    \ees
\end{proposition}
\begin{proof}
For $t\ge 0$, let $h_{t} $ be the heat kernel on $G/K$.
Now
\begin{align*}
    | \mathfrak{l}_*^\mu(f)(x) |^2 &\le \sup_{j \in \mathbb{Z}_+} | \mathcal M_{2^{-j}}^\mu(f)(x) |^2\\
    &\le \sup_{j \in \mathbb{Z}_+} | \mathcal M_{2^{-j}}^\mu(f)(x) - m_{2^{-j}}^\mu(0) \, f * h_{2^{-2j}}(x) |^2 
    + \sup_{j \in \mathbb{Z}_+} | m_{2^{-j}}^\mu(0)  \, f * h_{2^{-2j}}(x) |^2 \\
    &\le \sup_{j \in \mathbb{Z}_+} | \mathcal M_{2^{-j}}^\mu(f)(x) - m_{2^{-j}}^\mu(0) \, f * h_{2^{-2j}}(x) |^2 
    +  \sup_{t > 0} |m_t^\mu(0) \,f * h_t(x) |^2
\end{align*} 
From Lemma \ref{lemma:estimate for m_t^mu(lambda)} (i), we know that $m_t^\mu(0)$ is bounded for all $t>0$. Also, $f\mapsto \sup_{t>0}|(f*h_t)(x)|$ is bounded operator on $L^p(G/K)$ for $1<p\le\infty$ (\cite[Theorem 4.4.4]{AnkerJi1999}). On the other hand, using Plancherel's theorem for each $j \in \Z_+$ we have,
\begin{align*}
    \| \mathcal M_{2^{-j}}^\mu(f) - m_{2^{-j}}^\mu(0) \,  f * h_{2^{-2j}} \|_2^2 &= \int_0^\infty \int_K | m_{2^{-j}}^\mu(\lambda) - m_{2^{-j}}^\mu(0) \,  e^{-2^{-2j}(\rho^2 + \lambda^2)} |^2 \, \\
    &\hspace{5cm}| \widetilde{f}(\lambda, k) |^2 \, |c(\lambda)|^{-2} \, d\lambda \, dk.
\end{align*}
Thus,
\begin{align*}
   \left \| \sup_{j\in\Z_+} |\mathcal M_{2^{-j}}^\mu(f) - m_{2^{-j}}^\mu(0) \,  f * h_{2^{-2j}}| \right\|_2^2 
   &\le \sum_{j=1}^\infty \| \mathcal M_{2^{-j}}^\mu(f) - m_{2^{-j}}^\mu(0) \,  f * h_{2^{-2j}} \|_2^2 \\
    &= \sum_{j=1}^\infty \int_0^\infty \int_K | m_{2^{-j}}^\mu(\lambda) - m_{2^{-j}}^\mu(0) \,  e^{-2^{-2j}(\rho^2 + \lambda^2)} |^2 \, \\
    &\hspace{5cm}| \widetilde{f}(\lambda, k) |^2 \, |c(\lambda)|^{-2} \, d\lambda \, dk \\
    &\le \left( \sup_{\lambda > 0} \sum_{j=1}^\infty | m_{2^{-j}}^\mu(\lambda) - m_{2^{-j}}^\mu(0) \,  e^{-2^{-2j}(\rho^2 + \lambda^2)} |^2 \right) \| f \|_2^2
\end{align*}
Therefore the $L^2$-boundedness of $\fl_*^\mu$ will follow if we show that for $\operatorname{Re}\mu > -\alpha_0-1/2$,

\bes
    \sup_{\lambda > 0} \sum_{j=1}^\infty | m_{2^{-j}}^\mu(\lambda) - m_{2^{-j}}^\mu(0) \,  e^{-2^{-2j}(\rho^2 + \lambda^2)} |^2
\ees
is finite. We divide the argument into two cases depending on whether $2^{-j}\lambda >1$ or $2^{-j}\lambda \le 1$.
If $2^{-j}\lambda > 1$ then from Lemma \ref{lemma:estimate for m_t^mu(lambda)} (i), (iii), we have
\bes
    |m_{2^{-j}}^\mu(\lambda)| \le C\,  (2^{-j}\lambda)^{-(\operatorname{Re }\mu+\alpha_0+\frac{1}{2})}
\ees
and
\bes
    \left| m_{2^{-j}}^\mu(0) \,  e^{-2^{-2j}(\rho^2 + \lambda^2)}  \right|^2 \le C \, (2^{-j} \lambda)^{-2(\operatorname{Re }\mu+\alpha_0+\frac12)}.
\ees
So,
\bes
    \sum_{\{j\in\Z_+\,:\, 2\le 2^j<\lambda\}} \left| m_{2^{-j}}^\mu(\lambda) - m_{2^{-j}}^\mu(0) \,  e^{-2^{-2j}(\rho^2 + \lambda^2)} \right|^2 \le C \sum_{\{j\in\Z_+\,:\, 2\le 2^j<\lambda\}} (2^{-j} \lambda)^{-2(\operatorname{Re }\mu+\alpha_0+\frac12)}
\ees
which is finite for all $\lambda>0$ provided $\operatorname{Re}\mu>-\alpha_0-1/2$.

Suppose $2^{-j}\lambda \le 1$. Now, from Lemma \ref{lemma:estimate for m_t^mu(lambda)} (i), (iii) and applying the mean value theorem we obtain,
\begin{align*}
     \sum_{\{j\in\Z_+\,:\,  2^j>\lambda\}} \left| m_{2^{-j}}^\mu(\lambda) - m_{2^{-j}}^\mu(0) \,  e^{-2^{-2j}(\rho^2 + \lambda^2)} \right|^2\\
     &\hspace{-5cm}\le C  \sum_{\{j\in\Z_+\,:\,  2^j>\lambda\}} \left( \left| m_{2^{-j}}^\mu(\lambda) - m_{2^{-j}}^\mu(0)\right|^2+   \left|m_{2^{-j}}^\mu(0) \, \left(e^{-2^{-2j}(\rho^2 + \lambda^2)} -1 \right) \right|^2 \right)\\
     &\hspace{-5cm}\le C \sum_{\{j\in\Z_+\,:\,  2^j>\lambda\}}  (2^{-j}\lambda)^2
\end{align*}
Thus,
\bes
    \sum_{\{j\in\Z_+\,:\,  2^j>\lambda\}} \left| m_{2^{-j}}^\mu(\lambda) - m_{2^{-j}}^\mu(0) \,  e^{-2^{-2j}(\rho^2 + \lambda^2)} \right|^2,
\ees
is finite for all $\lambda>0$ provided $\operatorname{Re}\mu>-\alpha_0-1/2$. This completes the proof.

\end{proof}
We need the following proposition. 
\begin{proposition} \label{prop: l_*^mu is weak (1,1)}
    Suppose $\mu \in \C $ with $\operatorname{Re }\mu >0$. Then the operator $\fl_*^\mu $ is bounded from $L^1(G/K)$ to $L^{1,\infty}(G/K)$. In other words, the estimate
    \bes
        \left| \{ x\in G/K : \fl_*^\mu (f)(x) > \gamma \}\right| \le C \, \frac{\|f\|_1}{\gamma},
    \ees
    holds for all $\gamma>0$.
\end{proposition}
The proof of the proposition is exactly similar to the proof of \cite[Proposition 3.7]{WangZhang} and therefore we omit it's proof.

Now we are in a position to prove Theorem \ref{thm: L_*mu is bdd op}.    
\begin{proof}[Proof of Theorem \ref{thm: L_*mu is bdd op}] Using Marcinkiewicz's interpolation theorem, we obtain from Proposition \ref{prop: l_*^mu is bdd on L^2} and Proposition \ref{prop: l_*^mu is weak (1,1)} that $\fl_*^\mu$ is bounded on $L^{1+r}(G/K)$ for any $0<r<1$ and for $\operatorname{Re }\mu>0$. From Proposition \ref{prop: g_*^mu is in L^p for 1<p<=2} we also have that $\fg_*^\mu$ is bounded on $L^{1+r}(G/K)$ for any $0<r<1$ and for $\operatorname{Re }\mu>0$. Hence we obtain that $\cL_*^\mu$ is a bounded operator on $L^{1+r}(G/K)$ for any $0<r<1$ and for $\operatorname{Re }\mu>0$. Similarly from Proposition \ref{prop: l_*^mu is bdd on L^2} and from Proposition \ref{prop: g_*^mu is in L^p for 1<p<=2} we deduce that $\cL_*^\mu$ is a bounded operator on $L^{2}(G/K)$ for $\operatorname{Re }\mu>-\alpha_0-1/2$. Now, using complex interpolation, we obtain that
\begin{center}
$\cL_*^\mu$ is a bounded operator on $L^{p}(G/K)$ for $1<p\le 2$ and for $\operatorname{Re }\mu>\left( \frac{1}{p} -1 \right)(2\alpha_0 + 1)$.    
\end{center}
We observe that
\bes
    \left( \frac{1}{p} -1 \right)(2\alpha_0 + 1) = 1-d + \frac{d-1}{p}.
\ees

From the proof of Theorem \ref{main-thm} we have, for $\operatorname{Re }\mu >0$ $\cL_*^\mu $ is a bounded operator on $L^\infty(G/K)$.  Now, we interpolate between $p=2$, $\operatorname{Re } \mu > - \alpha_0 -1/2$ and $p=\infty$, $\operatorname{Re } \mu >0$ to get
$$
\cL_*^\mu \text{ is a bounded operator on  } L^p(G/K) \text{ for }
2 \le p \le \infty  \text{ and } \operatorname{Re}\,\mu > -\frac{2\alpha_0+1}{p}.
$$
Again we observe that 
\bes
-\frac{2\alpha_0+1}{p}=\frac{1-d}{p}.
\ees
This completes the proof.
\end{proof}
Observe that for $ \mu=0$, we recover the usual lacunary spherical maximal function. That is,
$$
\mathcal{L_*}f(x) = \mathcal{L_*}^0 f(x) = \sup_{t \in \Lambda} \left|\mathcal M_{t}^{0} f(x) \right|  = \sup_{t \in \Lambda} \left| \int_K f(u k a_{t} 1) \, dk  \right|.
$$
Then we have the following corollary.
\begin{corollary}
    For $d\ge 2$, $\mathcal{L_*}$ is a bounded operator on $L^p(G/K)$ for $1<p\le \infty.$
\end{corollary}

\begin{remark}
It would be interesting to investigate whether the range of $\mu$ for which $\mathcal M^\mu$ is bounded can be extended in the case $2<p\leq\infty$, as in the work of Chen--Shen--Wang--Yan \cite{ChenShenWangYan2025}. 
   
\end{remark}

\section*{Declaration of Generative AI Use}

During the preparation of this manuscript, we used AI-assisted tools for language editing and to help verify certain mathematical arguments. We also used AI in developing the proof of Lemma 3.4. 

\section*{Acknowledgements}
The first author is supported by the PhD fellowship of IIT Bombay. The Second author is supported partially by the ANRF, MATRIX grant no. ANRF/ARGM/2025/000319/MTR.

\end{document}